\documentclass[12pt]{article}

\usepackage{amsmath}
\usepackage{amsthm}
\usepackage{amssymb}
\usepackage{cite}
\usepackage{url}
\usepackage[multiple]{footmisc}
\usepackage{graphicx}
\usepackage{epstopdf}
\usepackage{chngcntr}
\usepackage{enumitem}
\usepackage{hhline}
\usepackage{array}
\usepackage{hyperref}
\usepackage{color}
\usepackage[normalem]{ulem}

\hypersetup{colorlinks=false}

\if@twoside
\else 
\fi

\numberwithin{equation}{section}
\counterwithin{figure}{section}
\counterwithin{table}{section}

\theoremstyle{plain}
\newtheorem{theorem}{Theorem}[section]
\newtheorem{lemma}[theorem]{Lemma}
\newtheorem{proposition}[theorem]{Proposition}

\newtheorem{corrolarry}[theorem]{Corollary}

\theoremstyle{definition}
\newtheorem{definition}{Definition}[section]

\theoremstyle{remark}
\newtheorem*{remark}{Remark}

\newcommand{\norm}[1]{\left\|#1\right\|}
\newcommand{\abs}[1]{\left\vert#1\right\vert}
\newcommand{\spr}[1]{\left\langle\,#1\,\right\rangle}
\newcommand{\kl}[1]{\left(#1\right)}
\newcommand{\Kl}[1]{\left\{#1\right\}}

\newcommand{\R}{\mathbb{R}} 
\newcommand{\C}{\mathbb{C}}

\newcommand{\Z}{\mathbb{Z}}

\newcommand{\OA}{{\Omega_A}}
\newcommand{\Ol}{{\Omega_l}}
\newcommand{\OT}{{\Omega_T}}
\newcommand{\OTcl}{{\cl \, \Omega_T}}
\newcommand{\OTclg}{{\clg \, \Omega_T}}
\newcommand{\OAaghl}{{\OA(\aghl)}}

\newcommand{\Lt}{{L_2}}

\newcommand{\LtOl}{{L_2(\Ol)}}
\newcommand{\LtOT}{{L_2(\OT)}}
\newcommand{\LtOTcl}{{L_2(\OTcl)}}
\newcommand{\LtOTclg}{{L_2(\OTclg)}}
\newcommand{\LtOA}{{L_2(\OA)}}

\newcommand{\jk}{{jk}}

\newcommand{\ag}{\alpha_g}
\newcommand{\agx}{\alpha_g^x}
\newcommand{\agy}{\alpha_g^y}
\newcommand{\aghl}{\ag h_l}

\newcommand{\vphi}{\varphi}

\newcommand{\cl}{{c_{l}}}
\newcommand{\clg}{c_{l,g}}
\newcommand{\hLGS}{h_{\text{LGS}}}

\newcommand{\GLGS}{G_{\text{LGS}}}
\newcommand{\GNGS}{G_{\text{NGS}}}

\newcommand{\citeNR}{\cite{Neubauer_Ramlau_2017}\,}

\newcommand{\At}{\tilde{A}}

\newcommand{\wjk}{w_{jk}}

\newcommand{\wjkl}{w_{jk,l}}
\newcommand{\wjklg}{w_{jk,lg}}
\newcommand{\wjklgt}{\tilde{w}_{jk,lg}}

\newcommand{\om}{\omega}

\newcommand{\psit}{\tilde{\psi}}

\newcommand{\pl}{\phi_l}

\newcommand{\pjk}{\phi_{\jk}}
\newcommand{\pjkl}{\phi_{\jk,l}}
\newcommand{\pjklg}{\phi_{\jk,lg}}

\newcommand{\Ml}{{M_l}}

\newcommand{\vpjkg}{\vphi_{\jk,g}}
\newcommand{\vpjk}{\vphi_{\jk}}

\newcommand{\et}{\tilde{e}}
\newcommand{\lt}{{\ell_2}}

\newcommand{\Atjk}{\At_{jk}}

\newcommand{\ujkn}{u_{jk,n}}
\newcommand{\ujkm}{u_{jk,m}}
\newcommand{\vjkn}{v_{jk,n}}
\newcommand{\vjkm}{v_{jk,m}}
\newcommand{\sjkn}{\sigma_{jk,n}}
\newcommand{\sjkg}{\sigma_{jk,g}}
\newcommand{\rjk}{r_{jk}}

\newcommand{\AtD}{\At^\dagger}
\newcommand{\AD}{\mathcal{A}}
\newcommand{\phiD}{\phi^\dagger}

\newcommand{\Ajk}{A_{jk}}

\newcommand{\D}{\mathcal{D}}
\newcommand{\Olm}{\Omega_{l,m}}

\newcommand{\en}{\vec{e}_n}

\newcommand{\Nlg}{N_{l,g}}

\newcommand{\wjklgm}{w_{jk,lgm}}

\newcommand{\pjklgm}{\phi_{\jk,lgm}}

\newcommand{\pjkg}{\phi_{jk,g}}
\newcommand{\Ajkg}{A_{\jk,g}}
\newcommand{\Ft}{\tilde{F}}

\title{A Frame Decomposition of the Atmospheric Tomography Operator}
\author{
Simon Hubmer\footnote{Johann Radon Institute Linz, Altenbergerstra{\ss}e 69, A-4040 Linz, Austria, (simon.hubmer@ricam.oeaw.ac.at), Corresponding author.} ,
Ronny Ramlau\footnote{Johannes Kepler University Linz, Institute of Industrial Mathematics, Altenbergerstra{\ss}e 69, A-4040 Linz, Austria, (ronny.ramlau@jku.at)} \footnote{Johann Radon Institute Linz, Altenbergerstra{\ss}e 69, A-4040 Linz, Austria, (ronny.ramlau@ricam.oeaw.ac.at)}
}

\begin{document}

\maketitle

\begin{abstract}
We consider the problem of atmospheric tomography, as it appears for example in adaptive optics systems for extremely large telescopes. We derive a frame decomposition, i.e., a decomposition in terms of a frame, of the underlying atmospheric tomography operator, extending the singular-value-type decomposition results of \citeNR by allowing a mixture of both natural and laser guide stars, as well as arbitrary aperture shapes. Based on both analytical considerations as well as numerical illustrations, we provide insight into the properties of the derived frame decomposition and its building blocks. 

\end{abstract}


\section{Introduction}

The imaging quality of earthbound astronomical telescopes like the Extremely Large Telescope (ELT) \cite{ELT_2020} of the European Southern Observatory (ESO), currently under construction in the Atacama desert in Chile, suffers from aberrations due to turbulences in the atmosphere, which result in blurred images. This is commonly counteracted by the use of \emph{Adaptive Optics} (AO) systems, which use the measurements of one or more \emph{Wavefront Sensors} (WFS) to suitably adjust \emph{Deformable Mirrors} (DMs) in such a way that the incoming wavefronts are corrected (flattened) after reflection on the mirrors; see Figure~\ref{fig_AO_system} (left). Since the atmosphere is constantly changing, this has to be done in real time. For more details on adaptive optics we refer to \cite{Roddier_1999,Roggemann_Welsh_1996,Ellerbroek_Vogel_2009}. 

There are a number of different types of AO systems, the simplest one being \emph{Single Conjugate Adaptive Optics} (SCAO). Thereby, the light of a so-called \emph{Natural Guide Star} (NGS), a bright star in the vicinity of an object of interest, is used to adjust the DM to obtain a corrected wavefront and thus a sharp image of the NGS and the nearby object. See Figure~\ref{fig_AO_system} (right) for a schematic drawing of an SCAO system. 

\begin{figure}
	\centering
	\includegraphics[width=0.45\textwidth]{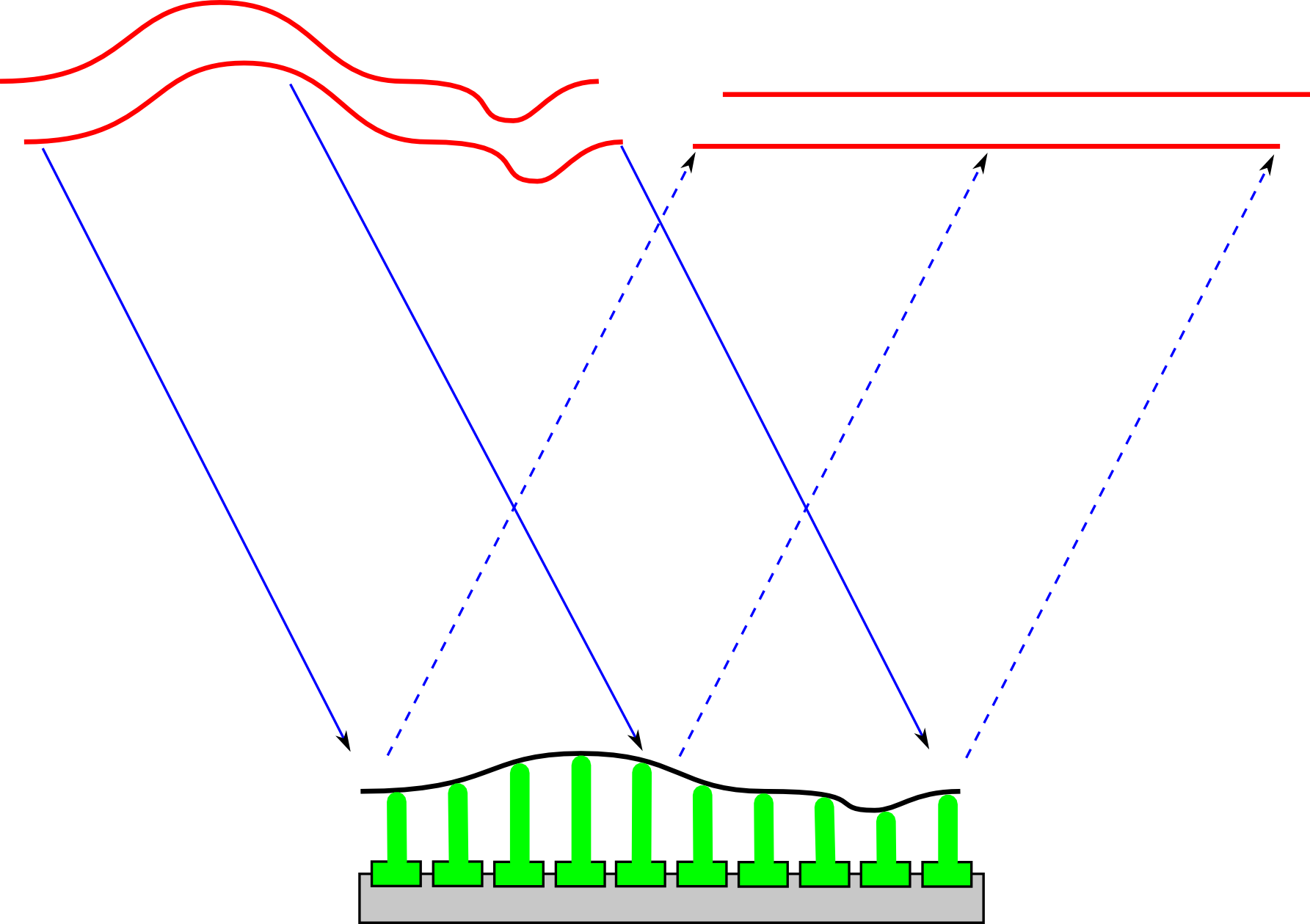}
	\qquad\qquad
	\includegraphics[width=0.40\textwidth, trim={6cm 0cm 6cm 0cm}, clip ]{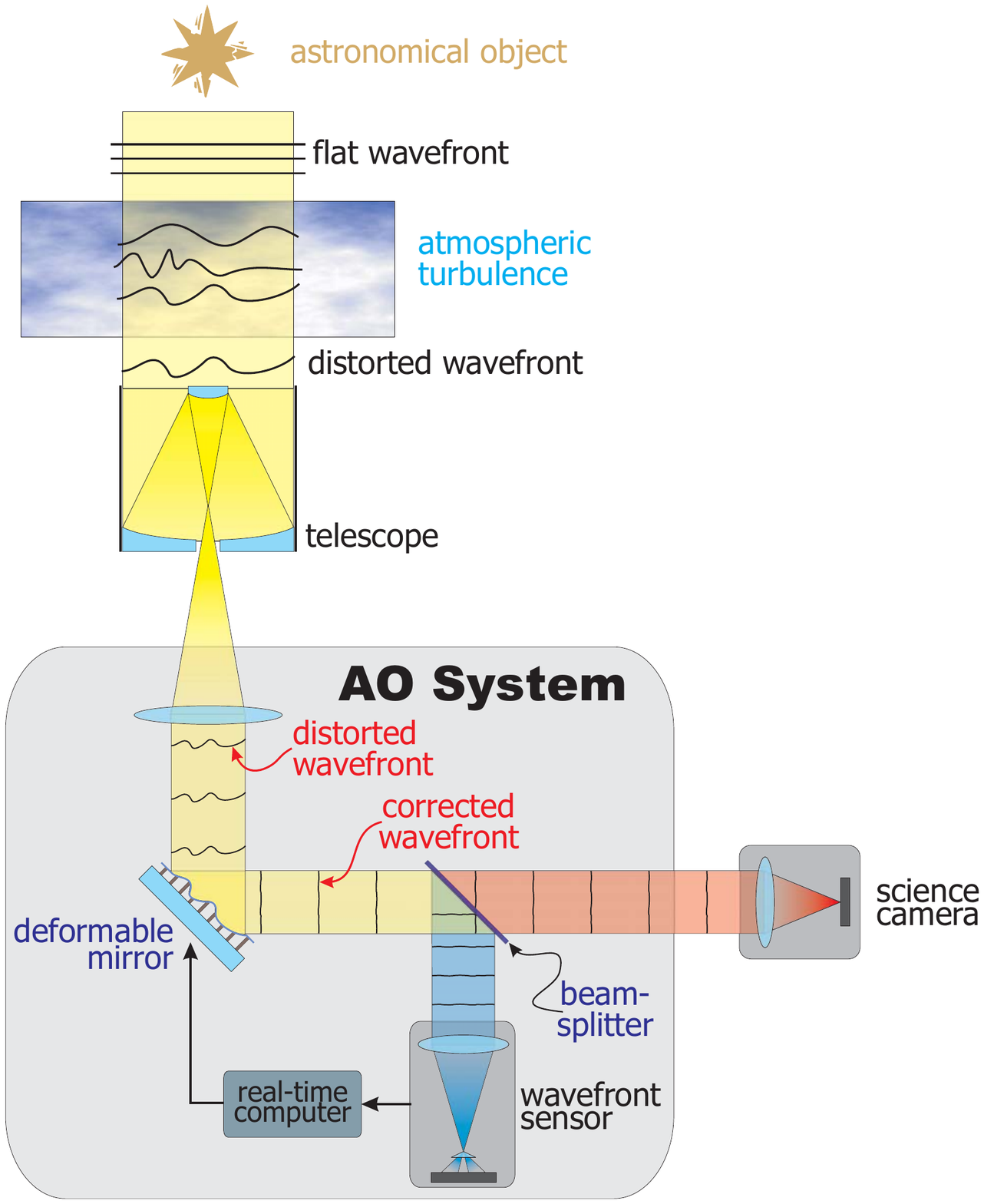}
	\caption{Correction of an incoming wavefront by a deformable mirror (left, image taken from \cite{Auzinger_2017}) and sketch of an SCAO system (right, image taken from \cite{Egner_2006}).}
	\label{fig_AO_system}
\end{figure}

In case that there is no bright enough NGS available in the vicinity of the object of interest, or if one wants to achieve a good correction over either a larger or multiple fields of view, one needs to resort to different, more complex AO systems. A good unidirectional correction in the absence of a NGS in the vicinity of an object of interest is for example achieved by \emph{Laser Tomography Adaptive Optics} (LTAO), while a correction over a large or multiple fields of view is achieved by \emph{Multiconjugate Adaptive Optics} (MCAO) and \emph{Multiobject Adaptive Optics} (MOAO), respectively \cite{Diolaiti_et_al_2016,Andersen_Eikenberry_Fletcher_Gardhuose_Leckie_Veran_Gavel_Clare_Jolissaint_Julian_Rambold_2006,Hammer_Sayede_Gendron_Fusco_Burgarella_Cayatte_Conan_Courbin_Flores_Guinouard_2002,Puech_Flores_Lehnert_Neichel_Fusco_Rosati_Cuby_Rousset_2008,Rigaut_Ellerbroek_Flicker_2000}. These are schematically depicted in Figure~\ref{fig_AO_systems}. In common with all of those different AO systems is the use of so-called \emph{Laser Guide Stars} (LGS), artificial stars created by powerful laser beams in the sodium layer of the atmosphere, which are used to increase the number of guide stars available for correction and thus enhance the imaging quality.

\begin{figure}
	\centering
	\includegraphics[width=1\textwidth, clip, trim={0cm 0cm 0cm 0cm}]{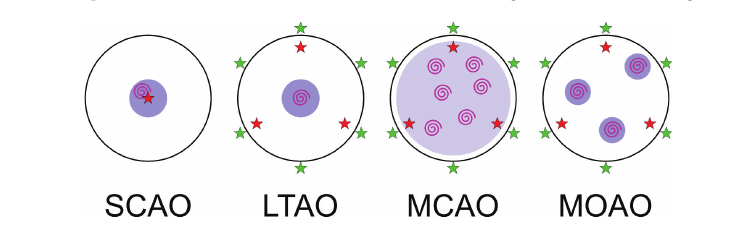}
	\caption{Different fundamental types of AO systems. Magenta spirals stand for astronomical objects of interest, and red and green stars for (the location of) NGS and LGS, respectively. The light blue areas correspond to those directions of view, for which the AO systems aim at achieving a correction. Image taken from \cite{Auzinger_2017}.}
	\label{fig_AO_systems}
\end{figure}

Since an SCAO system succeeds at enhancing the imaging quality in one direction of view only, measurements from a single WFS are enough to compute suitable correction shapes of the DM, because wavefront aberrations from two objects close to each other, in this case from the reference NGS and the considered object of interest, are approximately the same. However, this is not the case for LTAO, MCAO, and MOAO, where the NGS and LGS are far away from the object of interest, or a good correction has to be achieved over a large or multiple fields of view. Hence, one has to use multiple WFSs (typically one for each guide star) and DMs, whose correcting shapes have to be determined from the turbulence profile of the atmosphere, which in turn has to be calculated from the WFS measurements. This gives rise to the problem of the atmospheric tomography. 

Unfortunately, and in particular since the separation of the NGSs and LGSs is low (e.g., $1$ arcmin for MCAO and $3.5$ arcmin for MOAO), the problem of atmospheric tomography falls into the category of limited-angle tomography, which is known to be a severely ill-posed problem \cite{Davison_1983, Natterer_2001}. In addition, the number of available guide stars is relatively small as well (e.g., $6$ LGSs for the ELT), which in combination with the severe ill-posedness makes the reconstruction of the full atmospheric turbulence above the telescope a hopeless endeavour. Hence, one works with the commonly accepted assumption that the atmosphere contains only a limited number of turbulent layers, which are infinitely thin and located at predetermined heights. The problem of atmospheric tomography then becomes the task of reconstructing the turbulences on only a finite number of turbulence layers from the available WFS measurements. For a schematic depiction with three layers, (natural) guide stars, and WFSs see Figure~\ref{fig_AO_Tomo} (left).

\begin{figure}
	\centering
	\includegraphics[width=0.45\textwidth]{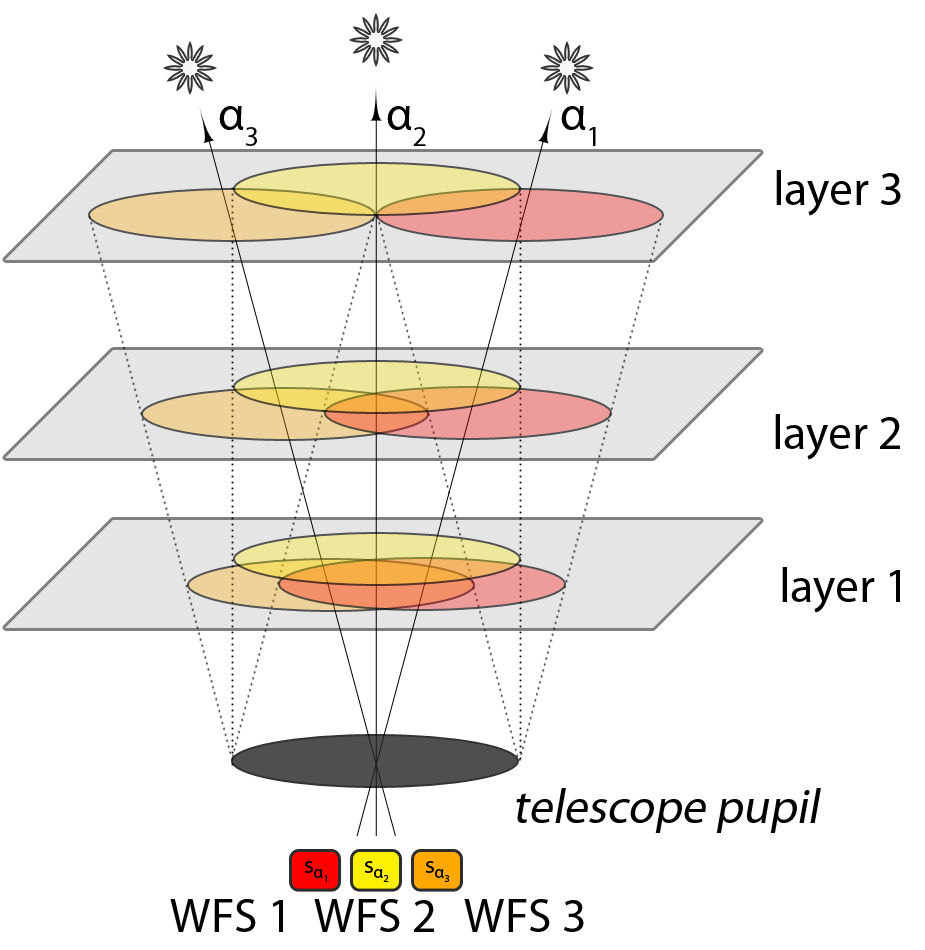}
	\qquad
	\includegraphics[width=0.45\textwidth]{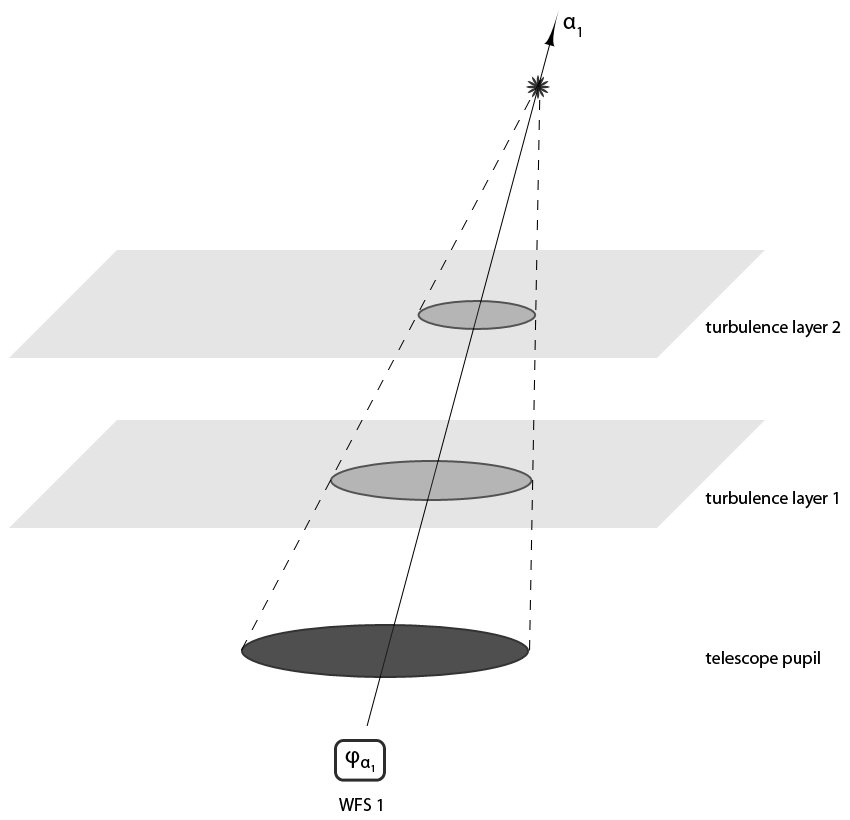}
	\caption{Schematic depiction of an atmospheric tomography setup with three turbulence layers, (natural) guide stars, and wavefront sensors (left). The coloured areas are those parts of the turbulence layers which are ``seen'' by each of the different wavefront sensors. Illustration of the cone effect (right) for a single LGS. Images taken from \cite{Yudytskiy_2014}.}
	\label{fig_AO_Tomo}
\end{figure}

A number of numerical reconstruction approaches have been proposed and developed for the atmospheric tomography problem, among them a minimum mean square error method \cite{Fusco_Conan_Rousset_Mugnier_Michau_2001}, a back-projection
algorithm \cite{Gavel_2004}, conjugate gradient type
iterative reconstruction methods with suitable preconditioning \cite{Ellerbroek_Gilles_Vogel_2003,Gilles_Ellerbroek_Vogel_2003,Gilles_Ellerbroeck_2008,Yang_Vogel_Ellerbroek_2006,Vogel_Yang_2006}, the Fractal Iterative Method (FrIM) \cite{Tallon_TallonBosc_Bechet_Momey_Fradin_Thiebaut_2010,Tallon_Bechet_TallonBosc_Louarn_Thiebaut_Clare_Marchetti_2012,Thiebaut_Tallon_2010}, the Finite Element Wavelet Hybrid Algorithm (FEWHA) \cite{Yudytskiy_2014,Yudytskiy_Helin_Ramlau_2013,Yudytskiy_Helin_Ramlau_2014,Helin_Yudytskiy_2013}, and Kaczmarz iteration \cite{Ramlau_Rosensteiner_2012, Rosensteiner_Ramlau_2013}; see also \cite{Ellerbroek_Gilles_Vogel_2002,Gilles_Ellerbroek_Vogel_2002,Gilles_Ellerbroek_Vogel_2007,Poettinger_Ramlau_Auzinger_2019,Raffetseder_Ramlau_Yudytskiy_2016,Ramlau_Obereder_Rosensteiner_Saxenhuber_2014,Saxenhuber_Ramlau_2016} and the references therein. All of these methods work comparatively well, each with its own peculiar advantages and drawbacks, and the resulting reconstructions have been successfully used to enhance the overall imaging quality of the corresponding AO systems. However, these numerical approaches themselves do not provide any deeper insight into the atmospheric tomography problem itself.

Hence, the authors of \citeNR set out to provide a mathematical analysis of the atmospheric tomography operator (defined below) underlying the problem, which is derived from the Radon transform using the layered and limited-angle structure of the problem \cite{Ellerbroek_Vogel_2009,Fusco_Conan_Rousset_Mugnier_Michau_2001}. In particular, they derived a singular-value-type decomposition of the operator, which not only provides the basis for efficient numerical reconstruction methods but also allowed to gain insight into the ill-posedness of the problem itself. In contrast to the already known singular-value decomposition of the limited-angle tomography operator \cite{Davison_1983,Natterer_2001}, the singular values of this decomposition could be computed explicitly.

However, the singular-value-type decomposition derived in \citeNR is only valid under a very restrictive set of assumptions. In particular, leaving aside some technicalities, it is only valid for square aperture shapes and tomography settings with only NGSs and no LGSs. This is obviously problematic for two reasons: firstly, the aperture shapes of telescopes is usually not square, and secondly, as we have already seen above, most of the AO systems which rely on atmospheric tomography include LGSs as an integral part of their design. Hence, many of the interesting theoretical results derived in \citeNR no longer hold for those (practically) important cases. Furthermore, while numerical routines based on this singular-value-type decomposition can in principle be adapted via measurement extension to (partly) circumvent the restriction of square aperture shapes, an adaption to include also LGSs is not possible in any straightforward way. This is mainly due to the so-called \emph{cone effect}: since a LGS is created by a powerful laser beam inside the sodium layer of the atmosphere, it is, in contrast to an ``infinitely far away'' NGS, located at a finite height. Thus, light travelling from the LGS to the telescope pupil passes through larger areas in lower atmospheric layers than in higher ones; see Figure~\ref{fig_AO_Tomo} (right) for an illustration. Mathematically, this results in the addition of a layer and guide star dependent scaling parameter in the atmospheric tomography operator (see below), which causes a number of complications.

The aim of this paper is to overcome the restrictions of the singular-value-type decomposition in \citeNR of the atmospheric tomography operator noted above. In particular, we want to find a decomposition which allows both LGSs and arbitrary aperture shapes. This is done in two steps: First, we consider the case of a tomography setup with only LGSs and no NGSs. Setting aside for the moment considerations on the practicality of such a setup (e.g., the tip/tilt problem), this is at the same time both a completion and a natural extension of the results of \citeNR, and an ideal starting point for deriving the main results of this paper. Then, we provide a decomposition of the atmospheric tomography operator for general problem settings including both a mixture of NGSs and LGSs, as well as arbitrary aperture shapes. This decomposition is done in terms of a set of functions which form a frame, with important implications for both theoretical as well as numerical aspects of the tomography problem.

The outline of this paper is as follows. In Section~\ref{sect_mathematical_setting} we describe the precise mathematical setting of atmospheric tomography considered in this paper, and in Section~\ref{sect_Frames} we review some necessary material on frames in Hilbert spaces. In Section~\ref{sect_SVD} we derive the decompositions for the atmospheric tomography operator mentioned above, first in the case of square domains and LGSs only, and then for a mixture of NGSs and LGSs as well as arbitrary aperture shapes. Section~\ref{sect_numerical_results} presents some numerical results based on the obtained analytical derivations and is followed by a short conclusion in Section~\ref{sect_conclusion}.

\section{Mathematical Setting}\label{sect_mathematical_setting}

In this section, we describe the precise mathematical setting of the atmospheric tomography problem considered in this paper. For the sake of consistency, we mainly use the same notations as in \citeNR.

The atmospheric tomography problem is a limited-angle tomography problem with only finitely many directions of view $\ag$, $g=1,\dots,G$, where $G$ denotes the total number of NGSs and LGSs. The directions $\ag \in \R^2$ are such that if $(\agx,\agy,1) \in \R^3$ points from the center of the telescope to the guide star $g$, then $\ag= (\agx,\agy) \in \R^2$. We denote the telescope aperture with $\OA \subset \R^2$ and assume that the atmosphere contains $L$ layers, where each layer is a plane at height $h_l$, $l=1,\dots,L$ parallel to $\OA$.

Since we do not only consider NGSs but also LGSs in this paper, we need to, in particular, take the cone effect into account (see above). For this, we need to define the parameters $\clg$. Assuming that $\GNGS$ and $\GLGS$ denote the number of NGSs and LGSs, respectively, such that $G = \GNGS + \GLGS$, we set
	\begin{equation*}
		\clg
		=
		\begin{cases}
			1 \,, & g \in \{1,\dots,\GNGS\} \,,
			\\
			1- \frac{h_l}{\hLGS} \,, & g \in \{\GNGS+1,\dots,\GNGS + \GLGS\} \,,
		\end{cases}
	\end{equation*}
where $\hLGS$ denotes the height of the LGSs. Since $h_l < \hLGS$ for all $l=1,\dots,L$, we have that $\clg \leq 1$. Furthermore, since we assume that the $h_l$ are in ascending order and that $h_L < \hLGS$, we have that $\clg \geq 1 - \tfrac{h_L}{\hLGS} > 0$.

For every layer $l$, we can now define the domains
	\begin{equation*}
	\begin{split}
		\Ol := \bigcup\limits_{g = 1}^G  \OAaghl \,,
	\end{split}
	\end{equation*}
where
	\begin{equation*}
	\begin{split}
		\OAaghl := \Kl{r \in \R^2 \, : \, \frac{r - \aghl}{\clg}  \in \OA} \,.
	\end{split}
	\end{equation*}
The domains $\Ol$ are exactly those parts of the layers which are ``seen'' by the wavefront sensors (compare with Figure~\ref{fig_AO_Tomo}) and are therefore those parts of the atmosphere on which one can expect to reconstruct the atmospheric turbulences.

Denoting by $\phi = (\phi_l)_{l=1,\dots,L}$ the turbulence layers and by $\vphi = (\vphi_g)_{g=1,\dots,G}$ the incoming wavefronts, the atmospheric tomography operator can now by defined as follows:
	\begin{equation}\label{def_A}
	\begin{split}
		A \, : \, \D(A) :=&\prod\limits_{l=1}^L \LtOl \to \LtOA^G \,,
		\\
		\phi &\mapsto \vphi_g = (A\phi)_g(r)
		:=
		\sum\limits_{l=1}^{L}\phi_l( \clg r+\aghl)  \,,
		\qquad g = 1,\dots,G \,,
	\end{split}
	\end{equation}
where $\LtOA^G$ denotes the cartesian product space $\prod_{g=1}^G \LtOA$. On the definition and image spaces of $A$ we define the canonic inner products
	\begin{equation}\label{def_spr}
		\spr{\phi,\psi} := \sum\limits_{l=1}^{L} \spr{\phi_l,\psi_l}_\LtOl \,,
		\qquad\spr{\vphi,\psi} := \sum\limits_{g=1}^{G} \spr{\vphi_g,\psi_g}_\LtOA \,.
	\end{equation}
Completely analogous to \cite[Theorem~3.1]{Neubauer_Ramlau_2017}, we have that the operator $A$ with respect to the above scalar products is not compact, and hence, a singular system does not necessarily need to exist for $A$.

However, the authors of \citeNR managed to derive a singular-value-type decomposition of what they called the \emph{periodic} atmospheric tomography operator $\At$, given by
	\begin{equation}\label{def_At}
	\begin{split}
		\At \, : \,  \LtOT^L &\to \LtOT^G \,,
		\\
		\phi &\mapsto \vphi_g = (\At\phi)_g(r)
		:=
		\sum\limits_{l=1}^{L}\phi_l( r+\aghl)  \,,
		\qquad g = 1,\dots,G \,,
		\end{split}
	\end{equation}
where $\OT := [-T,T]^2$ with $T$ sufficiently large, such that
	\begin{equation*}
		\OA + \aghl \subset \OT \,,
		\qquad
		g = 1, \dots,G \,,
		\quad
		l = 1,\dots,L \,,
	\end{equation*}
and under the assumption that functions in $\LtOT$ are periodic. This specific extension of the operator $A$ for the case of only NGSs allowed the use (of the special properties) of the functions
	\begin{equation}\label{def_wjk}
		\wjk(x,y) := \frac{1}{2 T}e^{ij\om x}e^{ik\om y} \,,
		\qquad
		\om := \frac{\pi}{T} \,,
	\end{equation}
which form an orthonormal basis of $\LtOT$. However, the derived decomposition introduces artefacts due to the periodicity assumptions and in particular does not cover the case of LGSs and mixtures of NGSs and LGSs. Furthermore, in practice, wavefront measurements are not given on the extended domain $\OT$ but on $\OA$ only. Thus, for applying the decomposition derived in \citeNR, these measurements have to be extrapolated to $\OT$ in some way, which is also not desirable.

Hence, in Section~\ref{sect_SVD}, we extend the singular-value-type decomposition from \citeNR by deriving a frame decomposition of the original operator $A$, based on the ideas of \citeNR but avoiding some of their shortcomings. For this, we first review some necessary materials on frames in Hilbert spaces below.

\section{Frames in Hilbert Spaces}\label{sect_Frames}

For the upcoming analysis, we need to recall some known results on frames in Hilbert spaces, which can for example be found in \cite{Daubechies_1992}. First, recall the definition of a frame:

\begin{definition}\label{def_frame}
A sequence $\{e_k\}_{k \in K}$ for some countable index set $K$ in a Hilbert space $H$ is a frame, if and only if there exist numbers $B_1,B_2 > 0$ such that for all $f \in H$ we have
	\begin{equation}\label{eq_frame}
		B_1 \norm{f}_H^2 \leq \sum\limits_{k \in K} \abs{\spr{f,e_k}_H}^2 \leq B_2 \norm{f}_H^2 \,.
	\end{equation}
The numbers $B_1,B_2$ are called \emph{frame bounds}. The frame is called \emph{tight} if $B_1 = B_2$.
\end{definition}

Now, for a given frame $\{e_k\}_{k\in K}$, one can consider the so-called \emph{frame (analysis) operator} $F$ and its adjoint $F^*$, or \emph{synthesis-operator}, defined by
	\begin{equation}\label{def_F_Fad}
	\begin{split}
		&F \, : \, H \to \lt(K) \,, \qquad
		f \mapsto \Kl{\spr{f,e_k}}_{k \in K} \,,
		\\ 
		&F^* \, : \, \lt(K) \to H \,, \qquad
		a_k \mapsto \sum\limits_{k \in K} a_k e_k  \,.
	\end{split}	
	\end{equation}
Due to \eqref{eq_frame} and the general fact that $\norm{F} = \norm{F^*} $, there holds that
	\begin{equation}\label{eq_bound_F_Fadj}
		\sqrt{B_1} \leq \norm{F} = \norm{F^*} \leq \sqrt{B_2} \,.
	\end{equation}
Furthermore, one can define the operator $S := F^*F$, i.e.,
	\begin{equation*}
		S f := \sum\limits_{k\in K} \spr{f,e_k} e_k \,,
	\end{equation*}
and it follows that $S$ is a bounded linear operator with $B_1 I \leq S \leq B_2 I$, where $I$ denotes the identity operator. Furthermore, $S$ is invertible and $B_2^{-1} I \leq S^{-1} \leq B_1^{-1} I$. It follows that if one defines $\et_k := S^{-1}e_k$, then the set $\{\et_k\}_{k\in K}$ also forms a frame, with bounds $B_2^{-1},B_1^{-1}$, which is called the \emph{dual frame} of $\{e_k\}_{k \in K}$. It is also known that every $f \in H$ can be written in the form
	\begin{equation}\label{eq_frame_rec}
		f = \sum\limits_{k\in K} \spr{f, \et_k} e_k
		= \sum\limits_{k\in K} \spr{f, e_k} \et_k \,.
	\end{equation}

It is not always possible to compute the dual frame $\et_k$ explicitly. However, since it holds that \cite{Daubechies_1992}
	\begin{equation}\label{eq_dual}
		\et_k = \frac{2}{A+B} \sum\limits_{j=0}^{\infty} R^j e_j \,,
	\end{equation}
where $R := I - \frac{2}{A+B} S$, the elements of the dual frame can be approximated by only summing up to a finite index $N$, i.e.,
	\begin{equation}\label{eq_dual_N}
		\et_k^N = \frac{2}{A+B} \sum\limits_{j=0}^{N} R^j e_j \,.
	\end{equation}
The induced error of this approximation is controlled by the frame bounds $A,B$, i.e.,
	\begin{equation}\label{eq_dual_approx_error}
		\norm{f - \sum\limits_{k \in K} \spr{f,e_k}\et_k^N} \leq \kl{\frac{B-A}{B+A}}^{N+1}\norm{f}\,.
	\end{equation}
For a numerical implementation, \eqref{eq_dual_N} can also be written in the following recursive form, which allows for an efficient numerical implementation in practice
	\begin{equation}\label{eq_dual_N_approx}
		\et_k^N = \frac{2}{A+B}e_k + R \et_k ^{N-1}\,,
	\end{equation}

Although for frames the decomposition of $f$ in terms of the functions $e_k$ is not unique, the representation in \eqref{eq_frame_rec} is the most \emph{economical}, in the sense the following
\begin{proposition}\label{prop_frame_min}
	If $f = \sum\limits_{k \in K} a_k e_k$ for some sequence $\{a_k\}_{k\in K} \in \lt(K)$ and if $a_k \neq \spr{f,\et_k}$ for some $k \in K$, then
		\begin{equation*}
			\sum\limits_{k \in K} \abs{a_k}^2 > \sum\limits_{k \in K} \abs{\spr{f,\et_k}}^2 \,.
		\end{equation*}
\end{proposition}
\begin{proof}
See for example \cite{Daubechies_1992}.
\end{proof}
Note that if $\{e_k\}_{k \in K}$ is a tight frame with bounds $B_1 = B_2 = B$, then we have that $S^{-1} = B^{-1} I$, $\et_k = e_k/B$ and therefore
	\begin{equation}\label{eq_frame_rec_tight}
		f = \frac{1}{B} \sum\limits_{k \in K} \spr{f, e_k} e_k\,,
		\qquad \text{and} \qquad
		\norm{f}^2 = \frac{1}{B} \sum\limits_{k\in K} \abs{\spr{f,e_k}}^2 \,.
	\end{equation}

Using these results, we now proceed to derive a frame decomposition of the atmospheric tomography operator $A$ below.

\section{Frame Decomposition}\label{sect_SVD}

In this section, we first derive a singular-value-type decomposition of the operator $\At$ in the case of only LGSs,  following the ideas of \citeNR. Afterwards, we derive a frame decomposition of the operator $A$ as defined in \eqref{def_A} where, in contrast to \citeNR, we do not restrict ourselves to only NGSs, but allow a mixture of both NGSs and LGSs, as well as (almost) arbitrary aperture shapes $\OA$.

\subsection{The Pure LGS Case}

In this section, we consider the periodic tomography operator \eqref{def_At} from \citeNR, but now for the case that instead of only NGSs, we consider a setting using only LGSs. Since in this case $\clg = \cl$ independent of the guide star direction $g$, the operator $\At$ can now be written in the form
	\begin{equation}\label{def_At_LGS}
	\begin{split}
		\At \, : \, \prod\limits_{l=1}^L \LtOTcl &\to \LtOT^G \,,
		\\
		\phi &\mapsto \vphi_g = (\At\phi)_g(r)
		:=
		\sum\limits_{l=1}^{L}\phi_l(\cl r+\aghl)  \,,
		\qquad g = 1,\dots,G \,,
		\end{split}
	\end{equation}
where, as before, $\OT = [-T,T]$, however now with $T$ sufficiently large, such that
	\begin{equation*}
		\OA + \aghl \subset \cl \, \OT \,,
		\qquad
		g=1,\dots,G\,, \quad l=1,\dots,L \,.
	\end{equation*}

We now derive a singular-value-type decomposition of the adapted operator $\At$ using the ideas from \citeNR. First, due to the presence of the constants $\cl$, it makes sense to use, for ever layer $l$, a different orthonormal basis of $\LtOTcl$, namely the functions
	\begin{equation}\label{def_wjkl}
		\wjkl(x,y) :=  \clg^{-1} \, \wjk\kl{(x,y)/\clg}
		\overset{\text{for only LGS}}{=}
		\cl^{-1} \, \wjk\kl{(x,y)/\cl}  	\,.
	\end{equation}
Any function $\pl \in \LtOTcl$ can be written in the form
	\begin{equation}\label{eq_exp_pl}
		\pl(x,y) = \sum\limits_{\jk \in \Z} \pjkl \, \wjkl(x,y) \,,
		\qquad
		\pjkl := \spr{\pl,\wjkl}_{\LtOl} \,,
	\end{equation}
and for $\phi = (\phi_l)_{l=1}^L \in \prod_{l=1}^L \LtOTcl $ we collect the expansion coefficients $\pjkl$ in the vectors $\pjk := (\pjkl)_{l=1}^L \in \C^L$. With this, we now get

\begin{proposition}
Let $\At$ be defined as in \eqref{def_At_LGS} and let the $G \times L$ matrices $\Atjk$ be defined by
	\begin{equation*}
		(\Atjk)_{gl} := (2T) \, \wjkl(\agx h_l, \agy h_l) \,.
	\end{equation*}
Then there holds
	\begin{equation}\label{eq_exp_A_LGS}
		(\At \phi)(x,y) = \sum\limits_{\jk \in \Z} \kl{\Atjk \pjk} \wjk(x,y) \,.
	\end{equation}
\end{proposition}
\begin{proof}
Using the definition \eqref{def_At_LGS} of $\At$ and the expansion \eqref{eq_exp_pl}, we get
	\begin{equation*}
	\begin{split}
		(\At \phi)_g(x,y) &= \sum\limits_{l=1}^L \phi_l(\cl x + \agx h_l, \cl y + \agy h_l)
		\\
		&=
		\sum\limits_{l=1}^L 		\sum\limits_{\jk \in \Z} \pjkl \, \wjkl(\cl x + \agx h_l, \cl y + \agy h_l)
		\\
		&=
		\sum\limits_{l=1}^L \sum\limits_{\jk \in \Z}
		\pjkl (2T) \wjk(x,y) \wjkl(\agx h_l, \agy h_l) \,,
	\end{split}
	\end{equation*}
from which the assertion immediately follows after interchanging the series, which is allowed since the norm of $\Atjk$ is bounded independently of $j,k$.
\end{proof}

As in \citeNR we consider the singular system for each of the matrices $\Atjk$, i.e., the vectors $\vjkn \in \C^L$, $\ujkn \in \C^G$ and numbers $\sjkn$, $n=1,\dots,\rjk \leq \min\{G,L\}$ such that
	\begin{equation}
	\begin{split}
		\Atjk \pjk = &\sum\limits_{n=1}^{\rjk} \sjkn \kl{\vjkn^H \, \pjk} \ujkn \,,
		\\
		\vjkm^H \vjkn =& \,\, \delta_{mn} \,,
		\quad
		\ujkm^H \ujkn = \delta_{mn} \,,
		\\
		\sigma_{jk,1} &\geq \dots \geq \sigma_{jk,\rjk} > 0 \,,
	\end{split}
	\end{equation}
where the superscript $H$ denotes the Hermitian adjoint (i.e. the complex transpose), $\rjk$ is the rank of the matrix $\Atjk$, and the $\sjkn^2$ are the positive eigenvalues of the matrices $\Atjk^H \Atjk$ and $\Atjk \Atjk^H$, respectively.

Hence, the decomposition of $\At$ in terms of the functions $\wjk$ is given by
	\begin{equation}\label{eq_dec_A_svd_LGS}
		(\At \phi)(x,y) = \sum\limits_{\jk \in \Z} \kl{\sum\limits_{n=1}^{\rjk} \sjkn \kl{\vjkn^H \, \pjk} \ujkn} \wjk(x,y) \,,
	\end{equation}
which is completely analogous to \cite[Equation~(10)]{Neubauer_Ramlau_2017}, however with a different singular systems $(\sjkn,\ujkn,\vjkn)$. Thus, for incoming wavefronts
	\begin{equation*}
	 \vphi = \sum\limits_{\jk \in \Z} \vpjk \wjk  \in \LtOT^G \,,
	 \qquad
	 \vpjk \in \C^G \,,
	 \quad
	\end{equation*}
the best-approximate solution of the equation
	\begin{equation*}
		\At \phi = \vphi
	\end{equation*}
is given by
	\begin{equation*}
		(\AtD \vphi)(x,y) :=  \sum\limits_{\jk \in \Z}  \kl{\sum\limits_{n=1}^{\rjk}  \frac{\kl{\ujkn^H \,\vpjk}}{\sjkn} \vjkn  } \wjkl(x,y) \,,
	\end{equation*}
where $\AtD$ denotes the Moore-Penrose generalized inverse of $\At$ (see for example \cite{Engl_Hanke_Neubauer_1996,Neubauer_Ramlau_2017}), which is well-defined if and only if the following \emph{Picard-condition} holds
	\begin{equation*}
		\sum\limits_{\jk \in \Z}  \sum\limits_{n=1}^{\rjk}  \frac{\abs{\ujkn^H \, \vpjk}^2}{\sjkn^2} < \infty \,.
	\end{equation*}

\subsection{The General Case}

In this section, we derive a frame decomposition for the general atmospheric tomography operator $A$ as defined in \eqref{def_A}, i.e., for the case of arbitrary aperture domains $\OA$ and both NGSs and LGSs.

The main idea of \citeNR, where some of the ideas of the upcoming analysis are taken from, was to use the special properties of the functions $\wjk$ \eqref{def_wjk}, in particular that they form an orthogonal basis of $\LtOA$. Unfortunately, for general domains $\OA$, the functions $\wjk$ do not necessarily  form a basis of $\LtOA$. However, they do form a (tight) frame in the sense of Definition~\ref{def_frame}, as we see in the following

\begin{lemma}\label{lem_frame_wjk}
Let $\wjk$ be defined as in \eqref{def_wjk}. If $T$ is large enough such that $\OA \subset \OT$, then the system $\{\wjk\}_{j,k\in\Z}$ forms a tight frame over $\LtOA$ with frame bound $1$, i.e.,
	\begin{equation*}
		\norm{\psi}_\LtOA^2 = \sum\limits_{j,k \in \Z}  \abs{\spr{\psi,\wjk}_\LtOA}^2 \,,
		\qquad
		\forall \, \psi \in \LtOA \,.
	\end{equation*}
\end{lemma}
\begin{proof}
We start by defining $\psit := \psi I_\OA$ where $I_\OA$ denotes the indicator function of $\OA$. Now since $\psit \in \LtOT$ and $\psit = \psi$ on $\OA$ and $\psit = 0$ on $\OT\setminus \OA$, we get
	\begin{equation*}
		\norm{\psi}_\LtOA^2
		=
		\norm{\psit}_\LtOT^2
		=
		\sum\limits_{\jk\in\Z}  \abs{\spr{\psit,\wjk}_\LtOT}^2
		=
		\sum\limits_{\jk\in\Z}  \abs{\spr{\psi,\wjk}_\LtOA}^2 \,,
	\end{equation*}
where we have used that since the $\wjk$ form an orthonormal basis of $\LtOT$, they are also a tight frame with frame bounds $B_1 = B_2 = 1$. This proves the assertion.
\end{proof}

Now, since $\{\wjk\}_{\jk \in \Z}$ forms a tight frame with bound $1$, it is also its own dual frame. Hence, \eqref{eq_frame_rec} implies that any function $\psi \in \LtOA$ can be written in the form
	\begin{equation}\label{eq_frame_rec_LtOA}
		\psi(x,y) = \sum\limits_{j,k \in \Z} \spr{\psi,\wjk}_\LtOA \wjk(x,y) \,.
	\end{equation}
In particular, we have that the incoming wavefronts $\vphi$ can be written in the form
	\begin{equation}\label{eq_vphi_dec}
		\vphi = \sum\limits_{\jk \in \Z} \vpjk \wjk \,,
		\qquad
		\vpjk := \kl{\vpjkg}_{g=1}^G :=  \kl{\spr{\vphi_g,\wjk}_\LtOA}_{g=1}^G \in \C^G \,.
	\end{equation}

We also want to expand functions on $\LtOl$ in terms of frames. It is not difficult to find frames for $\LtOl$;  for example, for large enough $T$, the sets $\Kl{\wjk}_{\jk\in\Z}$ or $\Kl{\wjkl}_{\jk\in\Z}$ already form frames over $\LtOl$. However, for the upcoming analysis, those frames do not satisfy the specific needs of the problem under consideration. Hence, we use a different, problem tailored frame, which we build from the functions
	\begin{equation}\label{def_frame_special}
	\begin{split}
		\wjklg(x,y) :=& \, \clg^{-1}\wjk((x,y)/\clg)  \, I_{\clg \OA + \aghl}(x,y)		 \,,
	\end{split}
	\end{equation}
with $\wjk$ defined as in \eqref{def_wjk}, where we choose $T$ large enough such that 
	\begin{equation}\label{cond_T}
		\OA \subset \OT
		\qquad
		\text{and}
		\qquad
		\Ol \subset \clg \OT
		\quad \forall \,  l \in \{1,\dots,L\}\,, \, g \in \{1,\dots,G \} \,,
	\end{equation}
which we assume to hold from now on. That these functions can indeed form frames can be seen from the following

\begin{proposition}\label{prop_frame_wjklg}
Let $\wjklg$ be defined as in \eqref{def_frame_special} and let $T$ be large enough such that \eqref{cond_T} holds. Then, for fixed $l$, the set $\Kl{\wjklg}_{\jk \in \Z, \, g=1,\dots,G}$ forms a frame over $\LtOl$ with frame bounds $1 \leq B_1 \leq B_2 \leq G$.
\end{proposition}
\begin{proof}
Let $\psi \in \LtOl$ be arbitrary but fixed and start by looking at
	\begin{equation}\label{eq_helper_1}
	\begin{split}
		&\sum\limits_{g=1}^{G}
		\sum\limits_{\jk \in \Z} \abs{\spr{\psi,\wjklg}_\LtOl}^2
		=
		\sum\limits_{g=1}^{G}
		\sum\limits_{\jk \in \Z}  \abs{\spr{\psi,\clg^{-1}\wjk((\cdot,\cdot)/\clg) \, I_{\clg \OA + \aghl}}_\LtOl}^2
		\\
		& \qquad \qquad =
		\sum\limits_{g=1}^{G}
		\sum\limits_{\jk \in \Z}  \abs{\spr{\psi \,  I_{\clg \OA + \aghl},\clg^{-1}\wjk((\cdot,\cdot)/\clg) }_\LtOTclg}^2 \,,
	\end{split}
	\end{equation}
where we have used that $\clg \OA + \aghl \subset \Ol \subset \clg\, \OT$. Now since for each (fixed) $g$ the functions $\clg^{-1}\wjk((x,y)/\clg)$ form an orthonormal basis of $\LtOTclg$ (and thus a tight frame with bound $1$), we get that
	\begin{equation*}
	\begin{split}
		&\sum\limits_{g=1}^{G}
		\sum\limits_{\jk \in \Z}  \abs{\spr{\psi \,  I_{\clg \OA + \aghl},\clg^{-1}\wjk((\cdot,\cdot)/\clg) }_\LtOTclg}^2
		=
		\sum\limits_{g=1}^{G} \norm{\psi \,  I_{\clg \OA + \aghl}}_\LtOTclg^2
		\\
		& \qquad =
		\sum\limits_{g=1}^{G} \int_{\clg\OT} \abs{\psi \,  I_{\clg \OA + \aghl}}^2
		= \sum\limits_{g=1}^{G} \int_{\clg \OA + \aghl} \abs{\psi }^2 \,,
	\end{split}
	\end{equation*}
which, together with \eqref{eq_helper_1}, yields
	\begin{equation}\label{eq_helper_2}
	\begin{split}
		\sum\limits_{g=1}^{G}
		\sum\limits_{\jk \in \Z} \abs{\spr{\psi,\wjklg}_\LtOl}^2
		= \sum\limits_{g=1}^{G} \int_{\clg \OA + \aghl} \abs{\psi }^2 \,.
	\end{split}
	\end{equation}
Now since $\Ol = \bigcup\limits_{g=1}^G \kl{\clg \OA + \aghl}$, we get that
	\begin{equation*}
	\begin{split}
		\norm{\psi}_\LtOl^2 =
		\int_{\Ol}  \abs{\psi }^2 \, \,
		\leq
		\sum\limits_{g=1}^{G} \int_{\clg \OA + \aghl} \abs{\psi }^2 \, \,
		\leq
		\sum\limits_{g=1}^{G} \int\limits_{\Ol}  \abs{\psi }^2 \, \,
		= G \norm{\psi}_\LtOl^2 \,.
	\end{split}
	\end{equation*}
Combing this together with \eqref{eq_helper_2}, we get that
	\begin{equation*}
	\begin{split}
		\norm{\psi}_\LtOl^2
		\leq
		\sum\limits_{g=1}^{G}
		\sum\limits_{\jk \in \Z} \abs{\spr{\psi,\wjklg}_\LtOl}^2
		\leq
		G \norm{\psi}_\LtOl^2 \,,
	\end{split}
	\end{equation*}
which proves the assertion.
\end{proof}

Now, denoting for fixed $l$ the dual frame of $\Kl{\wjklg}_{\jk \in \Z, \, g=1,\dots,G}$ by $\Kl{\wjklgt}_{\jk \in \Z, \, g=1,\dots,G}$, the above proposition together with \eqref{eq_frame_rec} implies that any function $\pl \in \LtOl$ can be written in the form
	\begin{equation}\label{eq_exp_pl_lg}
		\pl(x,y) =
		\sum\limits_{g=1}^{G}
		\sum\limits_{\jk \in \Z}
		\pjklg \wjklgt(x,y) \,,
		\qquad
		\pjklg = \spr{\pl,\wjklg}_\LtOl \,.
	\end{equation}

Next, we want to find an expansion of $A \phi$ in terms of the frame $\Kl{\wjk}_{\jk\in\Z}$. Due to \eqref{eq_frame_rec_LtOA} and since $A \phi \in \LtOA$, such an expansion is given by
	\begin{equation}\label{eq_exp_A}
		A \phi = \kl{(A \phi)_g}_{g=1}^G = \sum\limits_{\jk \in \Z} \kl{\spr{(A \phi)_g,\wjk}_\LtOA}_{g=1}^G \wjk \,,
	\end{equation}
As we already know, even though this might not be the only possible expansion, it is the most economical one in the sense of Proposition~\ref{prop_frame_min}. Furthermore, due to \eqref{eq_frame_rec_tight}, this expansion allows us to express $\norm{A \phi - \vphi}$ in terms of the expansion coefficients of $A \phi$ and $\vphi$, which is important for determining an (approximate) solution of $A\phi \,=\, \vphi$. We now derive an explicit expression for the coefficients $\spr{(A \phi)_g,\wjk}_\LtOA$ in terms of the coefficients $\pjklg$ in the following

\begin{proposition}\label{prop_Apg}
Let  $\phi = (\phi_l)_{l=1,\dots,L}$ with $\phi_l \in \LtOl$ for all $l=1,\dots,L$, and let  $A$, $\wjk$, and $\pjklg$ be defined as in \eqref{def_A}, \eqref{def_wjk}, and \eqref{eq_exp_pl_lg}, respectively. Then there holds
	\begin{equation}\label{eq_Aphi_sum}
		\spr{(A \phi)_g,\wjk}_\LtOA = (2T) \sum\limits_{l=1}^{L} \clg^{-1}\wjk(\aghl/ \clg)  \pjklg \,.
	\end{equation}
\end{proposition}
\begin{proof}
Due to the definition of $A$, and setting $r = (x,y)$ we have
	\begin{equation*}
	\begin{split}
		&\spr{(A \phi)_g,\wjk}_\LtOA
		=
		\sum\limits_{l=1}^{L} \spr{\pl(\clg \cdot + \aghl),\wjk}_\LtOA
 		=
		\sum\limits_{l=1}^{L} \int_{\OA} \pl(\clg r + \aghl) \overline{\wjk(r)} \, dr\,.
	\end{split}
	\end{equation*}
Using substitution in the integral yields
	\begin{equation*}
	\begin{split}
		&
		\int_{\OA} \pl(\clg r + \aghl) \overline{\wjk(r)} \, dr
		=
		\clg^{-2} \int_{\clg \OA + \aghl} \pl(r)\overline{ \wjk((r - \aghl)/ \clg)} \, dr
		\\
		& \qquad =
		\clg^{-2} (2T) \wjk(\aghl/ \clg)  \int_{\clg \OA + \aghl} \pl(r ) \overline{\wjk(r / \clg)}\, dr  \,.
	\end{split}
	\end{equation*}
Since $\clg \OA + \aghl \subset \Ol$, due to \eqref{def_frame_special} and \eqref{eq_exp_pl_lg}, we get
	\begin{equation*}
	\begin{split}
		&
		\clg^{-1} \int_{\clg \OA + \aghl} \pl(r ) \overline{\wjk(r / \clg)} \, dr
		=
		\clg^{-1}  \int_{\Ol} \pl(r) \overline{\wjk(r / \clg)  I_{\clg \OA + \aghl}(r)} \, dr
		=\pjklg \,.
	\end{split}
	\end{equation*}
Combining the above results, we get
	\begin{equation*}
	\begin{split}
		\spr{(A \phi)_g,\wjk}_\LtOA
 		&=
		\sum\limits_{l=1}^{L} \int_{\OA} \pl(\clg r + \aghl) \overline{\wjk(r)} \, dr
		\\		
		&=
		(2T)\sum\limits_{l=1}^{L}\clg^{-1}
		\wjk(\aghl/ \clg) \pjklg  \,.
	\end{split}
	\end{equation*}
which directly yields the assertion.
\end{proof}

Hence, combining \eqref{eq_exp_A} and the above proposition, we get that
	\begin{equation}\label{eq_exp_A_svd}
		(A \phi)(x,y)
		=
		\sum\limits_{\jk \in \Z}
		\kl{(2T)
		\sum\limits_{l=1}^{L}
		\clg^{-1}\wjk(\aghl/ \clg)  \pjklg
		}_{g=1}^G \wjk(x,y) \,,
	\end{equation}
Thus, if we define (with a slight abuse of notation), the vectors
	\begin{equation}\label{eq_def_phi_jk}
		\pjk :=  \kl{ \kl{\phi_{jk,l1}}_{l=1}^L
		, \dots , \kl{\phi_{jk,lG}}_{l=1}^L }  \in \C^{L\,\cdot G} \,,
	\end{equation}
and the $\C^{G \times (L \cdot G)}$ matrices
	\begin{equation}\label{def_Ajk}
		\Ajk :=
		(2T) \begin{pmatrix}
		\kl{c_{l,1}^{-1}\wjk\kl{\frac{\alpha_1 h_l}{c_{l,1}}} }_{l=1}^L &  0 & \cdots &  0\\
		0 & \kl{c_{l,2}^{-1}\wjk\kl{\frac{\alpha_2 h_l}{c_{l,2}}}}_{l=1}^L & \ddots&  \vdots \\
		\vdots & \ddots & \ddots &  0 \\
		0 & \cdots & 0& \kl{c_{l,G}^{-1}\wjk\kl{\frac{\alpha_G h_l}{c_{l,G}}}}_{l=1}^L
		\end{pmatrix} \,,
	\end{equation}
then we get the following expansion of the tomography operator
	\begin{equation}\label{eq_exp_A_svd_mat}
		(A \phi)(x,y)
		=
		\sum\limits_{\jk \in \Z}
		\kl{\Ajk \pjk }\wjk(x,y) \,.
	\end{equation}

The above expansion is obviously similar to \eqref{eq_exp_A_LGS} for the pure LGS case, however now with different (and slightly larger) matrices $\Ajk$. Hence, we can again consider the singular value decomposition of each of the matrices $\Ajk$, i.e., with another small abuse of notation, the vectors $\vjkn \in \C^{L\cdot G}$, $\ujkn \in \C^G$ and numbers $\sjkn$, $n=1,\dots,\rjk \leq G$ such that
	\begin{equation}\label{eq_SVD_ajk}
	\begin{split}
		\Ajk \pjk = &\sum\limits_{n=1}^{\rjk} \sjkn \kl{\vjkn^H\, \pjk} \ujkn \,,
		\\
		\vjkm^H \vjkn =& \,\, \delta_{mn} \,,
		\quad
		\ujkm^H \ujkn = \delta_{mn} \,,
		\\
		\sigma_{jk,1} &\geq \dots \geq \sigma_{jk,\rjk} > 0 \,,
	\end{split}
	\end{equation}
to get, in complete analogy to \eqref{eq_dec_A_svd_LGS}, the following frame decomposition of the atmospheric tomography operator $A$ defined in \eqref{def_A}:
	\begin{equation}\label{eq_dec_A_svd}
		(A \phi)(x,y) = \sum\limits_{\jk \in \Z} \kl{\sum\limits_{n=1}^{\rjk} \sjkn \kl{\vjkn^H\, \pjk} \ujkn} \wjk(x,y) \,.
	\end{equation}
Using this decomposition, we are now able to define the operator
 	\begin{equation}\label{def_AD}
 	\begin{split}
		(\AD \vphi)(x,y) &:= \sum\limits_{g=1}^{G}\sum\limits_{\jk \in \Z} \kl{ \kl{\sum\limits_{n=1}^{\rjk}  \frac{\kl{\ujkn^H\,\vpjk}}{\sjkn} \kl{\vjkn  }_{l+(g-1)L} }\wjklgt(x,y)}_{l=1}^L \,,
	\end{split}	
	\end{equation}
which we use to find a solution to the equation $A \phi = \vphi$ in Theorem~\ref{thm_svd_G1} below. Before we proceed to that though, we observe that the structure of the matrices $\Ajk$ allows to compute their singular-value-decomposition explicitly, which leads to the following

\begin{proposition}\label{prop_svd_expl}
Let $(\sjkn,\ujkn,\vjkn)_{n=1}^{\rjk}$, for $j,k \in \Z$ be the singular systems of the matrices $\Ajk$ as defined in \eqref{eq_SVD_ajk}. Then there holds
	\begin{equation}\label{eq_svd_expl}
		\rjk = G \,,
		\qquad
		\sjkn = \sqrt{\sum\limits_{l=1}^{L} c_{l,n}^{-2}}\,,
		\qquad
		\ujkn = \en \,,
	\end{equation}
	\begin{equation*}
		(\vjkn)_{i} = (2T) \sjkn^{-1}
		\begin{cases}
			c_{l,n}^{-1} \wjk\kl{\frac{-\alpha_{n}h_l}{c_{l,n}}}   &    \exists \, 1 \leq l \leq L : i = l + (n-1)L \,,
		    \\ 
			0   &  \text{else} \,,
		\end{cases} 
	\end{equation*}	
where $\en$ denotes the $n$-th unit vector in $\C^G$. 
\end{proposition}
\begin{proof}
Due to the structure of $\Ajk$, we have that
	\begin{equation*}
	\begin{split}
		\Ajk \Ajk^H 
		= (2T)^2 \, \text{diag} \kl{ \norm{\kl{\clg^{-1}\wjk(\alpha_g h_l/ c_{l,g})}_{l=1}^L}_\lt^2  }_{g=1,\dots,G}
		=
		\text{diag} \kl{ \sum\limits_{l=1}^{L} \clg^{-2}  }_{g=1,\dots,G} \,,
	\end{split}
	\end{equation*}
from which the formula for $\sjkn$, $\ujkn = \en$ and $\rjk = G$ immediately follow. Furthermore, since there holds 
	\begin{equation*}
		\vjkn = (1/\sjkn) \,\Ajk^H  \,\ujkn \,,
	\end{equation*}
the expression for $(\vjkn)$ immediately follows, which concludes the proof.
\end{proof}

Using the explicit representation derived above, we immediately get the following
\begin{corrolarry}\label{corr_AD}
The operator $\AD$ defined in \eqref{def_AD} can be written in the form
 	\begin{equation}\label{eq_AD_expl}
 	\begin{split}
		(\AD \vphi)(x,y) &:= (2T)\sum\limits_{g=1}^{G}\sum\limits_{\jk \in \Z} \kl{ \frac{\clg^{-1} }{\sjkg^2}
		\wjk\kl{-\frac{\aghl}{\clg}}   \kl{\vpjk}_g\wjklgt(x,y)}_{l=1}^L \,,
	\end{split}	
	\end{equation}
with the explicit form of the singular-values $\sjkg$ as derived in Proposition~\ref{prop_svd_expl}.	
\end{corrolarry}
\begin{proof}
Note that by Proposition~\ref{prop_svd_expl}, it follows that
	\begin{equation*}
		(\vjkn)_{l+(g-1)L} = (2T) \sjkn^{-1}
		\begin{cases}
			c_{l,n}^{-1} \wjk\kl{\frac{-\alpha_{n}h_l}{c_{l,n}}}   &    g = n \,,
		    \\ 
			0   &  \text{else} \,,
		\end{cases} 
	\end{equation*}	
from which, together with the expression for $\sjkn$ from Proposition~\ref{prop_svd_expl}, there follows
	\begin{equation}\label{eq_helper_4}
 	\begin{split}
		\sum\limits_{n=1}^{\rjk}  \frac{\kl{\ujkn^H\,\vpjk}}{\sjkn} \kl{\vjkn  }_{l+(g-1)L} 
		= 		
		(2T)  \frac{\kl{u_{jk,g}^H\,\vpjk}}{\sjkg^2}c_{l,g}^{-1} \wjk\kl{\frac{-\alpha_{g}h_l}{c_{l,g}}} 
		\,.
	\end{split}	
	\end{equation}
Furthermore, again due to Proposition~\ref{prop_svd_expl}, we get that
	\begin{equation}\label{eq_helper_11}
		u_{jk,g}^H\,\vpjk = (\en)^H\,\vpjk = (\vpjk)_g \,,
	\end{equation}		
which, together with \eqref{eq_helper_4} and the definition of $\AD$ immediately yields the assertion.
\end{proof}

Concerning the well-definedness of $\AD$, we can now derive the following
\begin{lemma}\label{lem_AD_welldef}
Let $\vphi \in \LtOA^G$ and let $\AD$ be defined as in \eqref{def_AD}. Then $\AD  \vphi$ is well-defined and there holds
	\begin{equation}\label{eq_bound_AD}
		\norm{\AD} \leq 
		\kl{\min\limits_{g=1,\dots,G} \Kl{ \sum\limits_{l=1}^L \clg^{-2} }}^{-\frac{1}{2}} 
		\leq
		1/\sqrt{L}\,.
	\end{equation}
\end{lemma}
\begin{proof}
Let $\Ft_{l}$ be the frame operator (compare with \eqref{def_F_Fad}) corresponding to the dual frame $\Kl{\wjklgt}_{\jk \in \Z\,, g=1,\dots,G}$ of $\Kl{\wjklg}_{\jk \in \Z\,, g=1,\dots,G}$  and let $\Ft^*_{l}$ be its adjoint. Since the dual frame is also a frame but with inverse frame bounds, it follows from Proposition~\ref{prop_frame_wjklg} together with \eqref{eq_bound_F_Fadj} that 
	\begin{equation*}
		1/\sqrt{G} \leq \norm{ \Ft_{l} } = \norm{\Ft^*_{l}} \leq  1 \,.
	\end{equation*}
Hence, for any sequence $a_l = \Kl{a_{\jk,lg}}_{\jk \in \Z \,, g=1,\dots,G} \in \lt$ there holds
	\begin{equation}\label{eq_dual_stability}
		\norm{\Ft^*_{l} a_l }_\LtOl^2 \leq \norm{ a_l}_\lt^2 \,.
	\end{equation}
which we now use for the choice
	\begin{equation}\label{def_al}
		a_l :=\Kl{ (2T)\frac{\clg^{-1} }{\sjkg^2}
		\wjk\kl{-\frac{\aghl}{\clg}}   \kl{\vpjk}_g}_{\jk \in \Z \,, g=1,\dots,G} \,,
	\end{equation}
where $\Kl{\vpjk}_{\jk\in\Z}$  are the expansion coefficients of $\vphi$ defined in \eqref{eq_vphi_dec}. Since due to \eqref{eq_AD_expl} and the choice of $a_l$ there holds 
	\begin{equation}\label{eq_AD_al}
		(\AD \vphi)_l = \Ft^*_{l} a_l \,,
	\end{equation}
we now obtain from \eqref{eq_dual_stability} that
	\begin{equation*}
	\begin{split}
		\norm{(\AD \vphi)_l}_\LtOl^2 
		&\leq 
		\sum\limits_{g=1}^G \sum\limits_{\jk \in \Z} \abs{(2T)\frac{\clg^{-1} }{\sjkg^2}
		\wjk\kl{-\frac{\aghl}{\clg}}   \kl{\vpjk}_g}^2 
		\\
		&=
		 (2T)^2 \sum\limits_{g=1}^G \sum\limits_{\jk \in \Z} \frac{\clg^{-2} }{\sjkg^4} \abs{\kl{\vpjk}_g}^2 		
		\,.
	\end{split}
	\end{equation*}
Now summing over $l$ we get that
	\begin{equation*}
	\begin{split}
		\sum\limits_{l=1}	^L		
		\norm{(\AD \vphi)_l}_\LtOl^2 
		\leq 
		(2T)^2 \sum\limits_{l=1}	^L \sum\limits_{g=1}^G \sum\limits_{\jk \in \Z} \frac{\clg^{-2} }{\sjkg^4} \abs{\kl{\vpjk}_g}^2 		
		\overset{\eqref{eq_svd_expl}}{=}		
		\sum\limits_{g=1}^G \sum\limits_{\jk \in \Z} \frac{\abs{\kl{\vpjk}_g}^2 	}{\sjkg^2} \,,
	\end{split}
	\end{equation*}
and therefore,
	\begin{equation*}
	\begin{split}
		\sum\limits_{l=1}	^L		
		\norm{(\AD \vphi)_l}_\LtOl^2 
		\leq 
		\max\limits_{\underset{g=1,\dots,G}{\jk \in \Z}}\Kl{\sjkg^{-2}} \sum\limits_{g=1}^G \sum\limits_{\jk \in \Z} \abs{\kl{\vpjk}_g}^2 	
		=
		\max\limits_{\underset{g=1,\dots,G}{\jk \in \Z}}\Kl{\sjkg^{-2}} \norm{\vphi}_{\LtOA^G}^2  \,,
	\end{split}
	\end{equation*}
where the last equality follows from the definition \eqref{eq_vphi_dec} of the coefficients $\vpjk$ together with Lemma~\ref{lem_frame_wjk}. Due to Proposition~\ref{prop_svd_expl} and since $0 < \clg \leq 1$ the singular values $\sjkg$ are independent of $j,k$ and bounded away from $0$. Hence, since $\vphi \in \LtOA^G$, it follows that $\AD \vphi$ is well-defined. Furthermore, the above estimate together with the explicit expression for $\sjkg$ from Proposition~\ref{prop_svd_expl} implies that
	\begin{equation*}
	\begin{split}
		\norm{\AD} \leq 
		\max\limits_{\underset{g=1,\dots,G}{\jk \in \Z}}\Kl{\sjkg^{-1}}
		=
		\kl{\min\limits_{g=1,\dots,G} \Kl{ \sjkg}	 }^{-1}
		=
		\kl{\min\limits_{g=1,\dots,G} \Kl{ \sqrt{\sum\limits_{l=1}^L \clg^{-2} }}}^{-1}\,,	
	\end{split}
	\end{equation*}
which, together with $0 < \clg \leq 1$, immediately yields the assertion.
\end{proof}

With this, we are now able to prove the main theorem of this section.
\begin{theorem}\label{thm_svd_G1}
Let $\vphi \in \LtOA^G$ be given and let $\Kl{\vpjk}_{\jk\in\Z}$  be its expansion coefficients defined as in \eqref{eq_vphi_dec}. Furthermore, let $(\sjkn,\ujkn,\vjkn)_{n=1}^{\rjk}$ for $j,k\in \Z$ be the singular systems of the matrices $\Ajk$ defined in \eqref{eq_SVD_ajk}. Moreover, for all $l \in \Kl{1,\dots,L}$ let $a_l$ be defined as in \eqref{def_al} and satisfy $a_l \in R(F_l)$, where $F_l$ denotes the frame operator corresponding to $\Kl{\wjklg}_{\jk \in \Z\,, g=1,\dots,G}$. Then the function $\AD \vphi$ is a solution of the equation $A \phi = \vphi$. Furthermore, among all other solutions $\psi \in \D(A)$ of $A\phi = \vphi$ there holds
	\begin{equation}\label{eq_phi_min}
		\sum\limits_{g,l=1}^{G,L} \sum\limits_{\jk \in \Z} \abs{\spr{(\AD \vphi)_l,\wjklg}_\LtOl}^2 \leq
		\sum\limits_{g,l=1}^{G,L} \sum\limits_{\jk \in \Z} \abs{\spr{\psi_l,\wjklg}_\LtOl}^2 \,,
	\end{equation}
Moreover, among all other possible expansions of $\AD \vphi$ in terms of the functions $\wjklgt$, \eqref{def_AD} is the most economical one in the sense of Proposition~\ref{prop_frame_min}.
\end{theorem}
\begin{proof}
Let $\vphi \in \LtOA^G$ and let $\vpjkg = \spr{\vphi_g,\wjk}_\LtOA$ be its canonical expansion coefficients, collected in the vectors $\vpjk$; compare with \eqref{eq_vphi_dec}. Furthermore, let $\phi \in \D(A)$ and let $\pjklg = \spr{\pl,\wjklg}_\LtOl$ be its canonical expansion coefficients, collected in the vectors $\pjk$; compare with \eqref{eq_exp_pl_lg} and \eqref{eq_def_phi_jk}. Due to Proposition~\ref{prop_Apg} there holds,
	\begin{equation}\label{eq_helper_3}
		\spr{(A \phi)_g,\wjk}_\LtOA \overset{\eqref{eq_Aphi_sum}}{=} (2T) \sum\limits_{l=1}^{L} \clg^{-1} \wjk(\aghl/ \clg)  \pjklg 
		\overset{\eqref{def_Ajk}}{=} (\Ajk\pjk)_g \,,
	\end{equation}
where the matrices $\Ajk$ are as in \eqref{def_Ajk}. Since due to Lemma~\ref{lem_frame_wjk} the set $\Kl{\wjk}_{\jk \in\Z}$ forms a tight frame over $\LtOA$ with frame bound $1$, it follows with \eqref{eq_frame_rec_tight} that
	\begin{equation*}
	\begin{split}
		&\norm{A \phi - \vphi}_\LtOA^2
		=
		\sum\limits_{g=1}^{G} \norm{(A \phi)_g - \vphi_g}_\LtOA^2
		\\
		& \qquad \overset{\eqref{eq_frame_rec_tight}}{=}
		\sum\limits_{g=1}^{G}
		\sum\limits_{\jk\in\Z}
		\abs{\spr{(A\phi)_g - \vphi_g,\wjk}_\LtOA}^2
		\overset{\eqref{eq_helper_3}}{=}
		\sum\limits_{g=1}^{G}
		\sum\limits_{\jk\in\Z}
		\abs{(\Ajk \pjk)_g - \vpjkg}^2 \,.
	\end{split}
	\end{equation*}
Hence, together with 
	\begin{equation*}
	\begin{split}
		\sum\limits_{g=1}^{G}
		\sum\limits_{\jk\in\Z}
		\abs{(\Ajk \pjk)_g - \vpjkg}^2
		=
		\sum\limits_{\jk\in\Z}
		\sum\limits_{g=1}^{G}
		\abs{(\Ajk \pjk)_g - \vpjkg}^2
		=
		\sum\limits_{\jk\in\Z}
		\norm{\Ajk \pjk - \vpjk}_{\lt}^2 \,,
	\end{split}
	\end{equation*}
we obtain
	\begin{equation}\label{eq_residual_equivalence}
		\norm{A \phi - \vphi}_\LtOA^2 =  
		\sum\limits_{\jk\in\Z}
		\norm{\Ajk \pjk - \vpjk}_{\lt}^2\,.
	\end{equation}
Thus, $\phi$ is a solution of $A \phi = \vphi$ if and only if its expansion coefficients $\pjk$ satisfy 
	\begin{equation}\label{eq_helper_5}
		\Ajk \pjk = \vpjk \,, \qquad \forall \, \jk \in \Z \,.
	\end{equation}
The solutions of these matrix-vector systems can be characterized via the singular systems $(\sjkn,\ujkn,\vjkn)_{n=1}^{\rjk}$ of the matrices $\Ajk$. Since in Proposition~\ref{prop_svd_expl} we showed that $\rjk = G$ and $\ujkn = \en$, it follows that at least one solution of \eqref{eq_helper_5} exists.  Hence, by the properties of the SVD it follows that the vectors
	\begin{equation}\label{eq_helper_7}
		\Ajk^\dagger \vpjk := \sum\limits_{n=1}^{\rjk}  \frac{\kl{\ujkn^H\,\vpjk}}{\sjkn} \vjkn \,,
	\end{equation}	
are the unique solution of \eqref{eq_helper_5} with minimal $\lt$ norm. Hence, if $\phi$ can be found such that $\pjk = \Ajk^\dagger \vpjk$, then for all other solutions $\psi$ of $A\phi = \vphi$ there holds
	\begin{equation}\label{eq_helper_6}
		\sum\limits_{g,l=1}^{G,L}  \abs{\spr{\phi_l,\wjklg}_\LtOl}^2 
		= 
		\norm{\phi_{jk}}_{\lt}
		\leq
		\norm{\psi_{jk}}_{\lt}
		=
		\sum\limits_{g,l=1}^{G,L} 
		 \abs{\spr{\psi_l,\wjklg}_\LtOl}^2 \,.
	\end{equation}
We now show that the choice $\phi := \AD \vphi$ is exactly such that
	\begin{equation}\label{eq_helper_8}
		\pjk = (\AD \vphi)_{\jk} = \Ajk^\dagger \vpjk 
		\,, \qquad \forall \, \jk \in \Z \,.
	\end{equation}
Since due to \eqref{eq_def_phi_jk} there holds $\pjklg = \kl{ \pjk }_{l+(g-1)L} $, this is equivalent to showing
	\begin{equation*}
		\kl{ \pjk }_{l+(g-1)L} = \kl{ \Ajk^\dagger \vpjk }_{l+(g-1)L}
		\,, \qquad \forall \, \jk \in \Z \,, \, 1\leq l\leq L \,, \, 1\leq g\leq G \,.
	\end{equation*}	
For this, note first that as in the proof of Corollary~\ref{corr_AD} we have that
	\begin{equation*}
	\begin{split}
		\kl{ \Ajk^\dagger \vpjk }_{l+(g-1)L}
		\overset{\eqref{eq_helper_7}}{=}
		\sum\limits_{n=1}^{\rjk}  \frac{\kl{\ujkn^H\,\vpjk}}{\sjkn} \kl{\vjkn  }_{l+(g-1)L} 
		\underset{\eqref{eq_helper_11}}{\overset{\eqref{eq_helper_4}}{=} }		
		(2T)  \frac{c_{l,g}^{-1}}{\sjkg^2}(\vpjk)_g \wjk\kl{\frac{-\alpha_{g}h_l}{c_{l,g}}} \,,
	\end{split}
	\end{equation*} 
and thus together with the definition of $a_l$ in \eqref{def_al} there holds
	\begin{equation}\label{eq_helper_10}
		\kl{ \Ajk^\dagger \vpjk }_{l+(g-1)L} = (a_l)_{\jk,g} \,.
	\end{equation} 
Hence, we now have to show that
	\begin{equation*}
		\pjklg =  (a_l)_{\jk,g}
		\,, \qquad \forall \, \jk \in \Z \,, \, 1\leq l\leq L \,, \, 1\leq g\leq G \,.
	\end{equation*}		
To do so, note that in \eqref{eq_AD_al} in the proof of Lemma~\ref{lem_AD_welldef} we saw that
	\begin{equation*}
		\phi_l = (\AD \vphi)_l = \Ft_l^* a_l \,, 
	\end{equation*}
where $\Ft_{l}$ is the frame operator of the dual frame $\Kl{\wjklgt}_{\jk \in \Z\,, g=1,\dots,G}$. Hence, we obtain
	\begin{equation*}
		F_l \pl = F_l (\AD \vphi)_l = F_l \Ft_l^* a_l = P_l a_l = a_l \,,
	\end{equation*}
where we used the fact (see \cite{Daubechies_1992}) that $F_l \Ft^*_l = P_l$, with $P_l$ denoting the orthogonal projector in $\lt$ onto $R(\Ft_l) = R(F_l)$, and that by assumption $a_l \in R(F_l)$. However, since this implies that 
	\begin{equation*}
		\pjklg
		= \spr{\pl,\wjklg}_\LtOl
		=
		\kl{F_l \pl}_{jk,g} = \kl{a_l}_{jk,g}   \,,
	\end{equation*}
it now follows that \eqref{eq_helper_8} holds, and thus $\AD \vphi$ is a solution of $A\phi = \vphi$. Together with \eqref{eq_helper_6}, this also implies the coefficient inequality \eqref{eq_phi_min}. Note that due to Lemma~\ref{lem_AD_welldef}, $\AD \vphi$ is well-defined since $\vphi \in \LtOA^G$. Finally, since $\pjklg = \spr{\pl,\wjklg}_\LtOl = \kl{a_l}_{jk,g}$, it follows that the expansion~\ref{def_AD} is the most economical one in the sense of Proposition~\ref{prop_frame_min}, which concludes the proof.
\end{proof}

\begin{remark}
In the above theorem we saw that $\AD \vphi$ is a solution of the atmospheric tomography problem given that $a_l \in R(F_l)$, which can be interpreted as a consistency condition. On the other hand, if instead we only assume that $A \phi = \vphi$ is solvable, then it follows from \eqref{eq_residual_equivalence} that for any solution $\phi$ there holds $\Ajk \pjk = \vpjk$, and thus that
	\begin{equation*}
		P_{N(\Ajk)^\perp} \pjk = \Ajk^\dagger \vpjk \,.
	\end{equation*}
Since the complete set of solutions is given by $\phiD + N(A)$, where $\phiD \in N(A)^\perp$ denotes the minimum-norm solution of $A\phi = \vphi$, and since by \eqref{eq_residual_equivalence} for any $\tilde{\phi} \in N(A)$ there holds $\tilde{\phi}_{jk} \in N(\Ajk)$, it follows that
	\begin{equation*}
		\Ajk^\dagger \vpjk = P_{N(\Ajk)^\perp} \phiD_{jk} \,. 
	\end{equation*}	
Noting that this can be rewritten as
	\begin{equation}\label{eq_HELPER}
		\Ajk^\dagger \vpjk 
		= \phiD_{jk} + \kl{P_{N(\Ajk)^\perp} - I } \phiD_{jk} 
		= \phiD_{jk} - P_{N(\Ajk)} \phiD_{jk}\,, 
	\end{equation}
and since due to \eqref{eq_AD_al} and \eqref{eq_helper_10} there holds
	\begin{equation*}
		(\AD \vphi)_l = \Ft_l^* \Kl{\kl{\Ajk^\dagger \vpjk }_{l+(g-1)L} }_{\jk\in\Z\,, g=1,\dots,G} \,,
	\end{equation*}
it now follows from \eqref{eq_HELPER} that
	\begin{equation*}
		\AD \vphi = \phiD + \kl{ \Ft_l^* b_l}_{l=1}^L \,,
		\qquad
		\text{where} 
		\qquad
		b_l := \Kl{\kl{P_{N(\Ajk)} \phiD_{jk}}_{l+(g-1)L} }_{\jk\in\Z\,, g=1,\dots,G} \,,
	\end{equation*}
which implies that if $a_l \notin R(F_l)$ one can still consider $\AD \vphi$ as an approximate solution.
\end{remark}
\vspace{2pt}

\begin{remark}
The explicit representation of $\AD$ given in \eqref{eq_AD_expl} can be used as the basis of an efficient numerical routine for (approximately) solving the atmospheric tomography problem $A \phi = \vphi$. Replacing the infinite sum over the indices $j,k$ by a finite sum, all one needs for implementation are the vectors $\vpjk$ and the evaluation of the functions $\wjklgt(x,y)$ on a predetermined grid. For this, note that since the functions $\{\wjklgt\}_{\jk\in\Z,g=1,\dots,G}$ form the dual frame of the functions $\{\wjklg\}_{\jk\in\Z,g=1,\dots,G}$, they can be efficiently numerically approximated via the approximation formula \eqref{eq_dual_N_approx}. Furthermore, the coefficients $\vpjk$ can be efficiently computed via \eqref{eq_vphi_dec} using the Fourier transform. Thus, since apart from $\vpjk$ all other quantities can be precomputed independently of the input data $\vphi$, the computation of $\AD \vphi$ via formula \eqref{eq_AD_expl} can be efficiently numerically implemented.
\end{remark}
\vspace{2pt}

\begin{remark}
In Lemma~\ref{lem_AD_welldef}, we have derived that $\norm{\AD} \leq 1/L$, i.e., that $\AD$ is bounded. In particular, for any noisy measurement $\vphi^\delta$ of the incoming wavefronts $\vphi$ this implies
	\begin{equation*}
		\norm{\AD \vphi - \AD \vphi^\delta}_{\D(A)} \leq 1/L \norm{\vphi - \vphi^\delta}_{\LtOA^G} \,,
	\end{equation*}
which shows that the (approximate) solution of the atmospheric tomography problem via the application of the operator $\AD$ is stable with respect to noise in the data. 

At first, this result seems counter-intuitive, since the atmospheric tomography operator is derived from  the Radon transform under the assumption of a layered atmosphere. Hence, since inverting the Radon transform is known to be an ill-posed problem, one expects $\AD$ to become unbounded with an increasing number of layers $L$. However, note that on the one hand $\AD y$ is only an approximate solution for $a_l \notin R(F_l)$, and on the other hand, for $L$ going to infinity, the atmospheric tomography operator $A$ as defined in \eqref{def_A} does not tend towards the (limited-angle) Radon transform. 

To see this, note that the sum in the definition of $A$ stems from the discretization of the line integrals of the Radon transform via the simple quadrature rule 
	\begin{equation*}
		\int\limits_0^1 f(x) \,dx \approx \sum\limits_{k=1}^L f(x_k) (x_k - x_{k-1} ) \,,
	\end{equation*}
where the weights $(x_k - x_{k-1})$ were dropped. Assuming that $(x_k - x_{k-1}) = 1/L$, which would correspond to equidistant atmospheric layers, the above quadrature rule becomes 
	\begin{equation*}
		\int\limits_0^1 f(x) \,dx \approx \frac{1}{L} \sum\limits_{k=1}^L f(x_k)  \,,
	\end{equation*}
which indicates that unless the atmospheric tomography operator is multiplied by $1/L$ in this case, it does not converge to the desired limit as the number of layers tends to infinity. Incidentally, if $A$ is replaced by $(1/L)A$, then $\AD$ has to be replaced by $L \AD$, and it then follows from \eqref{eq_bound_AD} that
	\begin{equation}
		\norm{L \AD } = L \norm{\AD } \leq \sqrt{L} \,,
	\end{equation}
resulting in an upper bound for the solution operator $L \AD$ which tends towards infinity for an increasing number of layers $L$, as one would expected. 

Note that further adaptations to the tomography operator $A$ would have to be made in order to ensure its convergence to the Radon transform in the limit for $L \to \infty$. These correspond to the integration weights having to be adapted to different guide star directions. However, since this is practically not relevant as $L$ is generally fixed, we do not go into further details here. In summary, we want to emphasize that the upper bound on $\AD$ does not contradict the ill-posedness of the general tomography problem.
\end{remark}
\vspace{2pt}

\begin{remark}
Equation \eqref{eq_phi_min} implies that $\AD \vphi$ can be seen as the \emph{minimum-coefficient} solution of $A \phi = \vphi$. Since the sets $\Kl{\wjklg}_{\jk \in \Z,g=1,\dots,G}$ do not form tight frames over $\LtOl$, $\AD \vphi$ is not necessarily also the minimum-norm solution of the equation. 

One way to obtain a minimum-norm solution is to take $\Ml$ piecewise disjoint domains $\Olm \subset \Ol$ satisfying
	\begin{equation*}
		\Ol = \bigcup\limits_{m=1}^{\Ml} \Olm \,,
		\qquad \text{and} \qquad
		\forall \, l,g \, \, \, \exists \, \Nlg \subset \Kl{1,\dots,\Ml}
		: \,
		\clg \OA + \aghl
		= \bigcup\limits_{m \in \Nlg} \Olm
		 \,,
	\end{equation*}
and to use, instead of $\Kl{\wjklg}_{\jk\in\Z,g=1,\dots,G}$, the set of functions
	\begin{equation*}
		\Kl{\wjklgm}_{\jk\in\Z,\,g=1,\dots,G, \,m=1,\dots,\Ml} \,,
	\end{equation*}
where
	\begin{equation*}
		\wjklgm(x,y) := \clg^{-1} \wjk((x,y)/\clg) I_{\Olm}(x,y)\,.
	\end{equation*}
Similarly as in the proof of Proposition~\ref{prop_frame_wjklg}, it can be shown that for fixed $l$, these functions form a tight frame with frame bound $G$ for $\LtOl$. Hence, in complete analogy to \eqref{eq_exp_pl_lg}, it is possible to decompose $\pl$ as follows:
	\begin{equation}\label{eq_exp_pl_lm}
		\pl(x,y) =
		\frac{1}{G}
		\sum\limits_{g=1}^{G}
		\sum\limits_{m=1}^{\Ml}
		\sum\limits_{\jk \in \Z}
		\pjklgm \wjklgm(x,y) \,,
		\qquad
		\pjklgm = \spr{\pl,\wjklgm}_\LtOl \,.
	\end{equation}
Following the same steps as in the proof of Proposition~\ref{prop_Apg}, one finds that
	\begin{equation*}
		\spr{(A \phi)_g,\wjk}_\LtOA = (2T) \sum\limits_{l=1}^{L} \clg^{-1}\wjk(\aghl/ \clg)  \sum\limits_{m\in \Nlg}\pjklgm \,,
	\end{equation*}
from which it follows that
	\begin{equation*}
		(A \phi)_g(x,y) = \sum\limits_{\jk\in \Z} \kl{ (2T) \sum\limits_{l=1}^{L} \clg^{-1}\wjk(\aghl/ \clg)  \sum\limits_{m\in \Nlg}\pjklgm } \wjk(x,y) \,.
	\end{equation*}
Defining (again with a small abuse of notation) the vectors
	\begin{equation}
	\begin{split}
		\pjk &:= ( \pjkg)_{g=1,\dots,G} \in \C^{G\cdot L\,\cdot M_1\cdot,\dots,\cdot M_L} \,,
		\\
		\pjkg &:=  \kl{ \kl{\phi_{jk,1gm}}_{m=1}^{M_1}
		, \dots , \kl{\phi_{jk,Lgm}}_{m=1}^{M_L} }  \in \C^{L\,\cdot M_1\cdot,\dots,\cdot M_L} \,, 
	\end{split}
	\end{equation}
we get, in analogy to \eqref{eq_exp_A_svd_mat}, that $A$ can be written as
	\begin{equation}
		(A \phi)(x,y)
		=
		\sum\limits_{\jk \in \Z}
		\kl{\Ajk \pjk }\wjk(x,y) \,,
	\end{equation}
but now with
	\begin{equation*}
		\Ajk = 
		\begin{pmatrix}
		A_{\jk,1} & 0 & \cdots & 0 
		\\
		0 & A_{\jk,2} & \ddots & \vdots
		\\
		\vdots & \ddots & \ddots & 0
		\\
		0 & \cdots & 0 & A_{\jk,G} 
		\end{pmatrix}
		\in \C^{G \times G\,\cdot L\,\cdot M_1\cdot,\dots,\cdot M_L} \,,
	\end{equation*}
where
	\begin{equation*}
	\begin{split}
		(\Ajkg) = 
		(2T) 
		\begin{pmatrix}
		c_{1,g}^{-1}\wjk(\alpha_g h_1/ c_{1,g}) \kl{I_{m \in N_{1,g}}}_{m=1}^{M_1}
		& 
		\cdots 
		&
		c_{L,g}^{-1}\wjk(\alpha_g h_L/ c_{L,g})  \kl{I_{m \in N_{1,g}}}_{m=1}^{M_L}
		\end{pmatrix}
		\\
		\in \C^{1\times L\,\cdot M_1\cdot,\dots,\cdot M_L} \,,
	\end{split}
	\end{equation*}
and
	\begin{equation*}
		I_{m\in\Nlg} :=
		\begin{cases}
			1 & m \in \Nlg \,, 
			\\
			0 & \text{else} \,.
		\end{cases}
	\end{equation*}
Denoting again by $(\sjkn,\ujkn,\vjkn)_{n=1}^{\rjk}$ the singular value decompositions of the matrices $\Ajk$, we arrive at the same decomposition of $A$ as in \eqref{eq_dec_A_svd}, but obviously with the new singular values and functions. All results derived above hold analogously for this new decomposition, with the difference that, due to the tightness of the frame, $\AD \vphi$ is then not only the  minimum-coefficient solution of $A \phi = \vphi$, but also the minimum-norm solution. Furthermore, one can again derive the singular values of $\Ajk$ explicitly, namely
	\begin{equation*}
		\sjkn = \sqrt{\sum\limits_{l=1}^{L}c_{l,n}^{-2}\sum\limits_{m=1}^{\Ml} I_{m\in N_{l,n}} } 
		=
		\sqrt{\sum\limits_{l=1}^{L}c_{l,n}^{-2}\abs{N_{l,n}} }\,,
	\end{equation*}
In addition, as in \ref{lem_AD_welldef} one can again show that $\AD \vphi$ is well-defined and that 
	\begin{equation*}
		\norm{\AD} \leq \kl{\min\limits_{g=1,\dots,G} \Kl{ \sjkg}	 }^{-1}
		= \kl{\min\limits_{g=1,\dots,G} \Kl{    \sum\limits_{l=1}^{L}c_{l,n}^{-2}\abs{N_{l,n}}  } }^{-\frac{1}{2}} \,.
	\end{equation*}

However, note that using this approach, the number of subdomains $\Olm$ can increase strongly with the number of guide stars $G$. Furthermore, note that the number of frames per layer considered in the two approaches presented in this section are directly related to $G$ and $\Ml$, respectively. Hence, since the used frames have discontinuities, using the approach based on the frames $\{\wjklgm\}$ potentially introduces more discontinuities into the solution in a numerical implementation than the approach based on the frames $\{\wjklg\}$, which is not necessarily desirable in practice.
\end{remark}
\vspace{2pt}

\begin{remark}
In some situations, it is desirable to, instead of the standard $\Lt$ inner product \eqref{def_spr}, work with the weighted inner product
	\begin{equation}\label{def_spr_weight}
		\spr{\phi,\psi} := \sum\limits_{l=1}^{L} \frac{1}{\gamma_l} \spr{\phi_l,\psi_l}_\LtOl \,,
	\end{equation}
where $\Kl{\gamma_l}_{l=1}^L$ denotes a normalized, nonzero sequence of weights. Although we do not focus on this issue in our current investigation, it should be noted that adapting the frame decomposition to include this weighted inner product should be straightforward.
\end{remark}

\section{Numerical Illustrations}\label{sect_numerical_results}

In this section, we present some numerical illustrations showcasing the building blocks of the frame decomposition of the atmospheric tomography operator derived above, namely the frame functions $\wjklg$ and their corresponding dual frame functions $\wjklgt$. While the functions $\wjklg$ have been defined explicitly and thus can be easily visualized, this is not the case for the functions $\wjklgt$, which are only implicitly defined. However, since they are the foundation of the frame decomposition, it is of interest and importance to get some idea about their visual representation, which we aim to do in this section.

It was already noted above that the dual frame functions $\wjklgt$ can be numerically approximated via formula \eqref{eq_dual_N_approx}, and that this iterative process can be implemented efficiently using the Fourier transform. However, the functions $\wjklgt$ are problem-adapted, i.e., they are dependent on the concrete parameter setting of the considered atmospheric tomography problem, which we thus first need to decide upon. Hence, for our purposes, we take the same setup as in \citeNR, which in turn is adapted from the MAORY setup \cite{MAORY_2020, Diolaiti_et_al_2016} as planned for the ELT currently being built by ESO. More precisely, we consider a telescope with a diameter of $42$ m (the originally planned ELT size), which leads to a circular aperture domain $\OA$, as well as $6$ NGS positioned uniformly on a circle with a diameter of $1$ arcmin, with the corresponding directions $\ag$ given in the following table ($\rho=0.000290888$, $a=\rho\cdot 0.5$, $b=\rho\cdot 0.866025$):
	\begin{align*}
	\begin{array}{|l|c|c|c|c|c|c|} 
		\hline
		g & 1 & 2 & 3 & 4 & 5 & 6 
		\\ \hline
		\ag & (\rho,0) & (a,b)& (-a,b) & (-\rho,0) & (-a,-b) & (a,-b) 
		\\ \hline
	\end{array}
	\end{align*}
	
For defining the layer profile we choose the ESO Standard Atmosphere, which consists of 9 layers located at the heights $h_l$ given in the following table: 
	\begin{equation*}
	\begin{array}{|c|c|c|c|c|c|c|c|c|c|} 
		\hline
		l & 1 & 2 & 3 & 4 & 5 & 6 & 7 & 8 & 9 
		\\ \hline
		h_l\text{(m)} & 0 & 140 & 281 & 562 & 1125 & 2250 & 4500 & 9000 & 18000 
		\\ \hline
	\end{array} 
	\end{equation*}

For the numerical implementation, and in particular for the approximation of the dual frame functions $\wjklgt$ via formula \eqref{eq_dual_N_approx}, which in our present case reads as
	\begin{equation}\label{eq_approx_wjklgt}
		\wjklgt^N =
		\frac{2}{1+G} \wjklg + \wjklgt ^{N-1} - \frac{2}{1+G} \sum\limits_{pq\in \Z} \sum\limits_{r=1}^G \spr{\wjklgt ^{N-1},w_{pq,lr}}_\LtOl w_{pq,lr} \,,
	\end{equation}
one can only consider a finite number of indices $j,k$ (and thus also of $p,q$). Hence, for our numerical illustrations, we restricted ourselves to $j,k,p,q \in \Kl{-63,\dots,63}$, which is motivated by the fact that the wavefront sensors of MAORY cannot reconstruct higher frequencies which might be present in the wavefronts anyway. For computing the dual frame functions $\wjklgt$ depicted below, we implemented the iterative procedure \eqref{eq_approx_wjklgt} in Matlab (R2019a) on a desktop computer running on a Mac OS with an 8 core processor (Intel Xeon W-2140B CPU@3.20GHz) and 32GB RAM. The iteration itself was stopped after $500$ iterations for each of the functions, which is more than sufficient as suggested by the error estimate \eqref{eq_dual_approx_error} and the fact that already after much fewer iterations the magnitudes of the updates become almost negligible. The computation time for a single dual frame function was more than two hours, although this can be significantly reduced to a couple of seconds by more sophisticated implementations (currently under development). Note again that the computation of the dual frame functions has to be done only once for each AO setting and can be carried out in advance of the actual atmospheric tomography reconstruction.

The resulting frame and dual frame functions $\wjklg$ and $\wjklgt$, for different values of $j,k,l,g$ are depicted in Figure~\ref{fig_adj_frame_layer_center} and Figure~\ref{fig_adj_frame_layer_border}, which show absolute value plots of the real part of those functions in linear scale and (for the dual frame functions) also in logarithmic scale. While the frame functions depicted in Figure~\ref{fig_adj_frame_layer_center} are living on the same layer ($l=2$) and differ only in the guide star direction ($g = 2$ and $g=5$), the frame functions depicted in Figure~\ref{fig_adj_frame_layer_border} vary also in the other parameters. In particular, they are living on atmospheric layers far apart from each other ($l=2$ and $l=9$). One can clearly  (especially from the log-plots) see that even though by definition the frame functions $\wjklg$ have local support, namely exactly the domains $\OAaghl$, the resulting dual frame functions $\wjklgt$ spread over the whole domain $\Ol$, which is due to the fact that all the frame functions $\wjklg$, also for all other values of $g$, contribute to their definition. The influence of all the different guide star directions on each of the dual frame functions $\wjklgt$ is also apparent in the appearance of ring-like structures corresponding to the the domains $\OAaghl$. Furthermore, the nicely visible periodicity of the frame functions $\wjklg$ due to their definition involving an exponential function manifests itself in a periodic pattern of the dual frame functions $\wjklgt$.

\begin{figure}
	\centering
	\includegraphics[width=\textwidth, trim = {12cm 20cm 10cm 18cm}, clip]{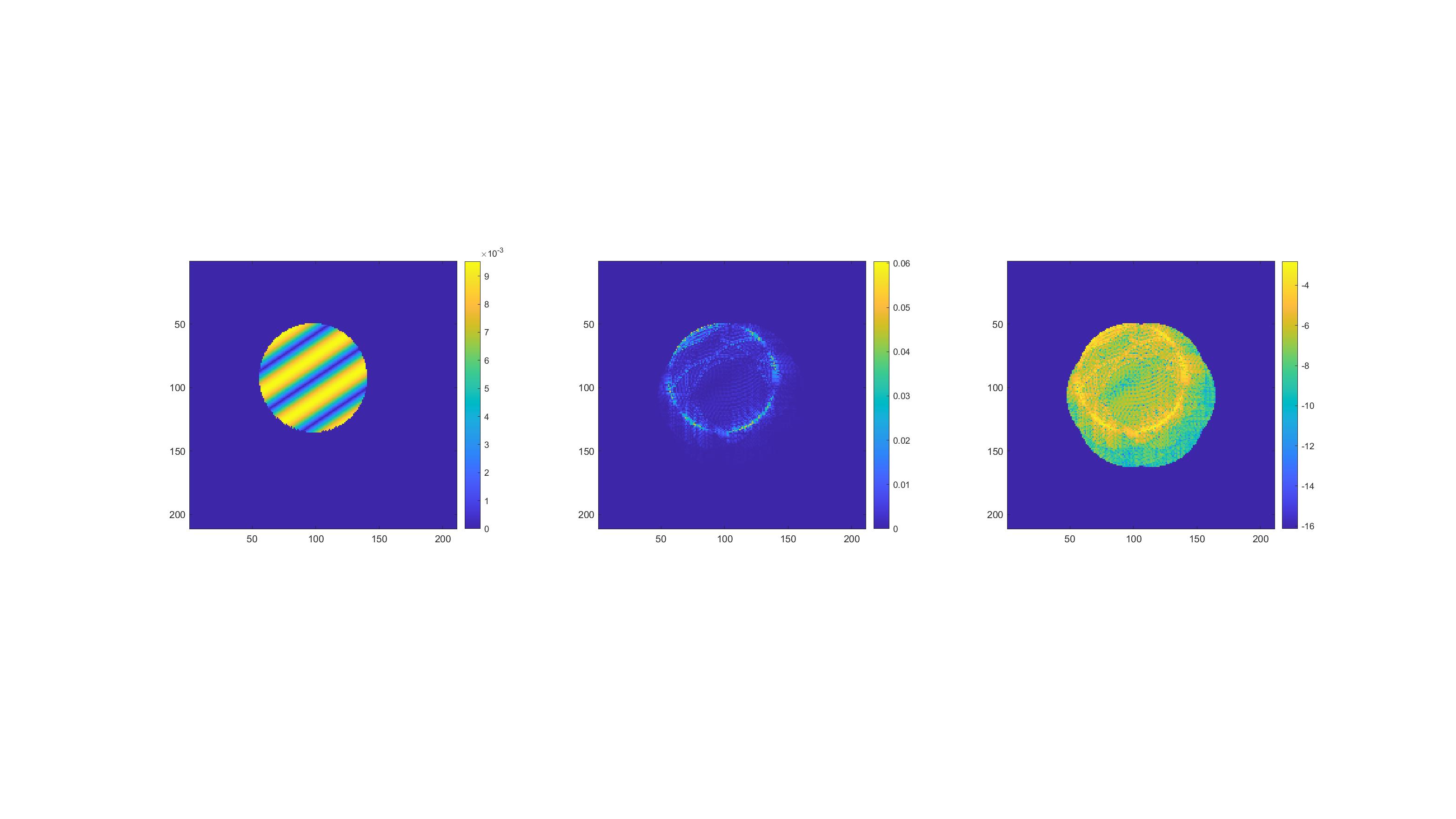}
	\\ \vspace{5pt}
	\includegraphics[width=\textwidth, trim = {12cm 20cm 10cm 18cm}, clip]{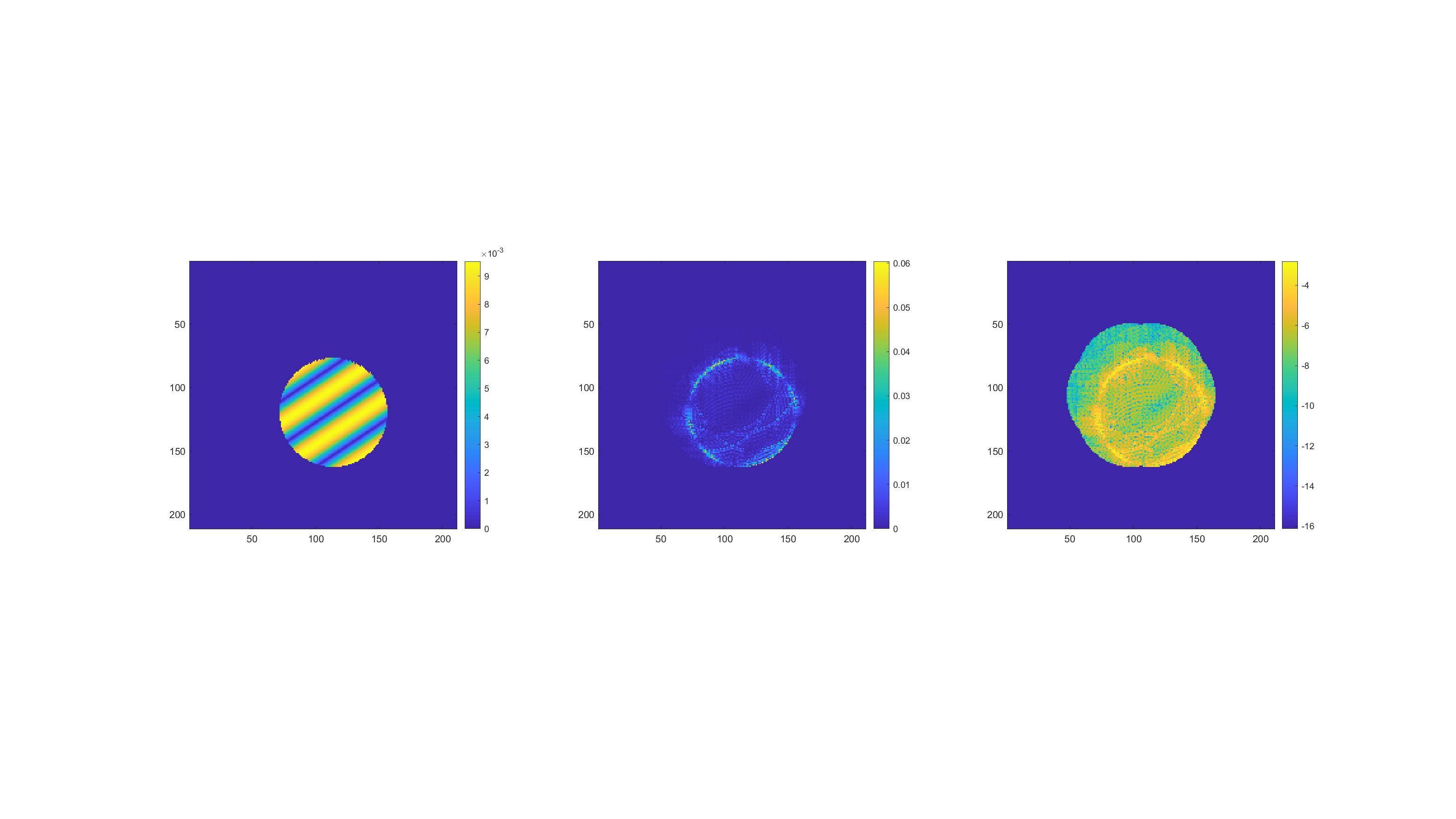}
	\caption{Frame functions $\wjklg$ (left) and dual frame functions $\wjklgt$ (middle, right) for $j = 3$, $k = 2$, $l=8$, with $g=2$ (top) and $g=5$ (bottom), respectively. Plotted is the absolute value of the real part in linear (left, middle) and logarithmic scale (right).}
	\label{fig_adj_frame_layer_center}
\end{figure}

\begin{figure}
	\centering
	\includegraphics[width=\textwidth, trim = {12cm 20cm 10cm 18cm}, clip]{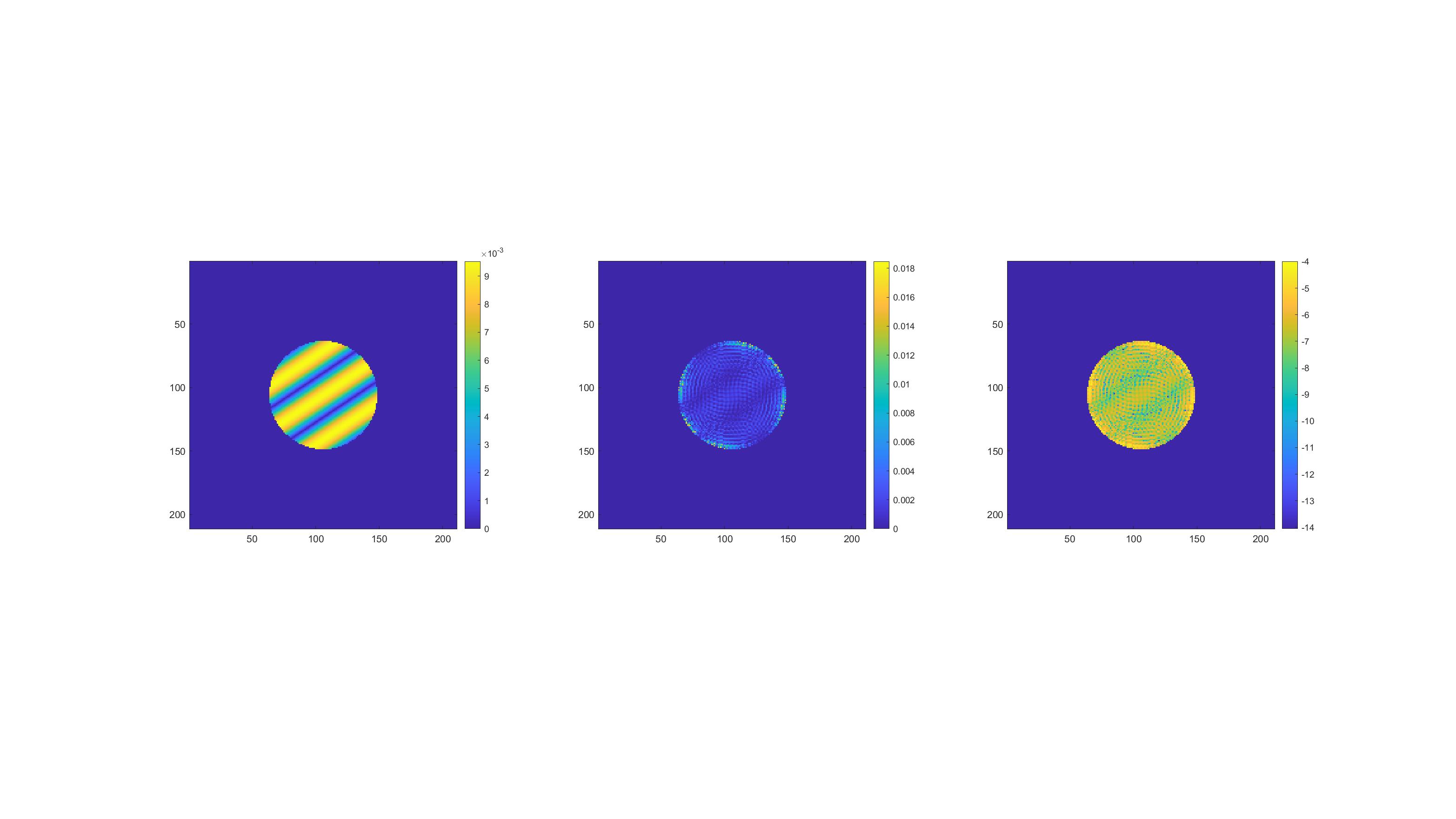}
	\\ \vspace{5pt}
	\includegraphics[width=\textwidth, trim = {12cm 20cm 10cm 18cm}, clip]{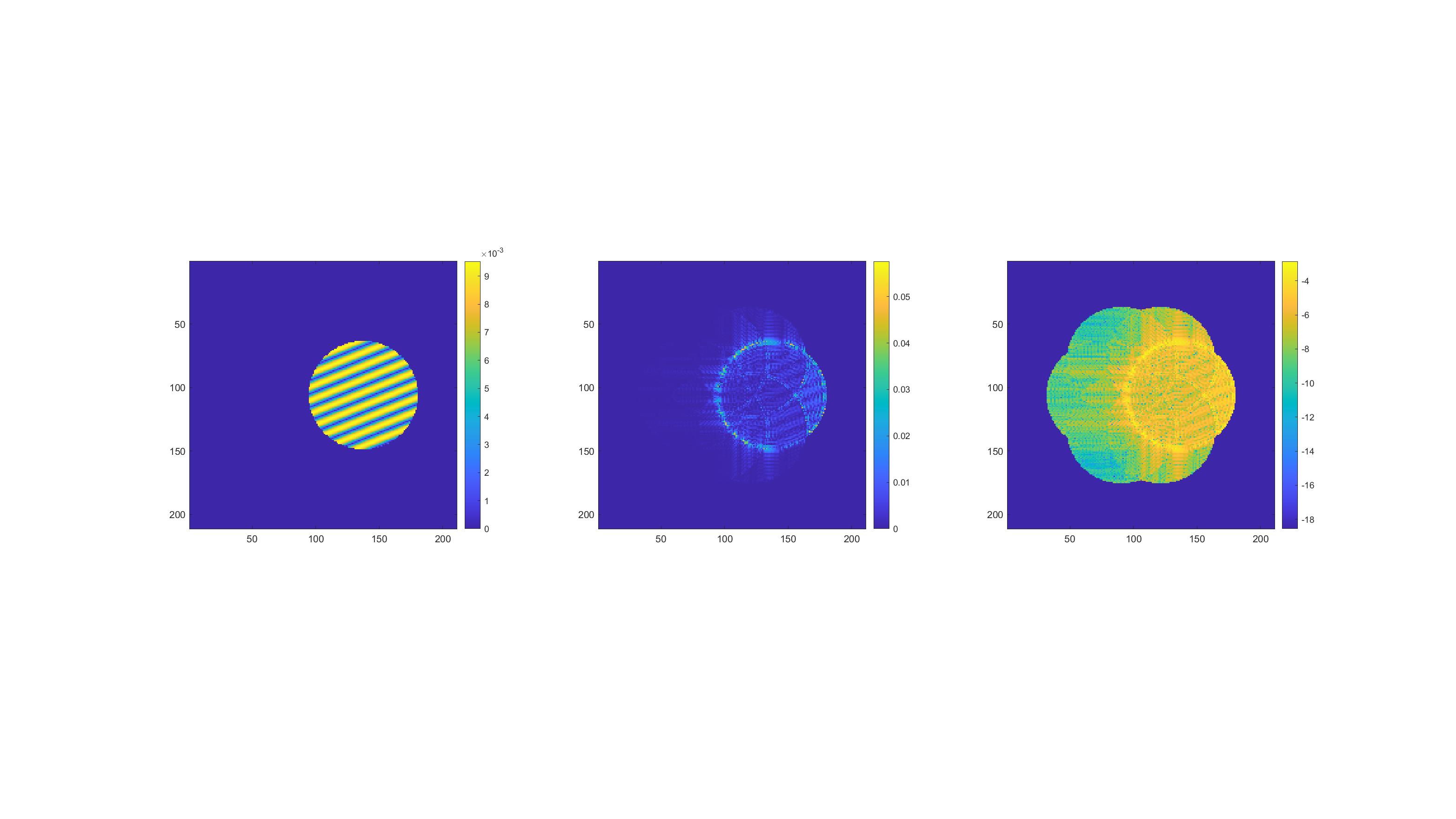}	
	\caption{Frame functions $\wjklg$ (left) and dual frame functions $\wjklgt$ (middle, right) for $j = 3$, $k = 2$, $l=2$, $g=2$ (top) and $j = 10$, $k = 4$, $l=9$, $g=4$ (bottom). Plotted is the absolute value of the real part in linear (left, middle) and logarithmic scale (right).}
	\label{fig_adj_frame_layer_border}
\end{figure}

\section{Conclusion and Outlook}\label{sect_conclusion}

In this paper, we considered the problem of atmospheric tomography by analysing the underlying tomography operator, and derived a frame decomposition for different problem settings. In particular, in contrast to \citeNR, we considered settings with a mixture of both natural and laser guide stars, and did not place any restrictions on the shape of the aperture shape. The resulting decomposition yields information on the behaviour of the operator and provides a stable reconstruction algorithm for obtaining a minimum-coefficient least-squares solution of the atmospheric tomography problem. Furthermore, we presented numerical illustrations showcasing some of the building blocks of the frame decomposition. In a forthcoming publication currently under preparation, we plan to present numerical reconstruction results from both simulated and experimental AO data. For this, we are in particular developing efficient, accurate, and stable numerical methods for the computation of the dual frame functions.

\section{Support}

S. Hubmer and R. Ramlau were (partly) funded by the Austrian Science Fund (FWF): F6805-N36. The authors would like to thank Dr.\ Stefan Kindermann for valuable discussions on some theoretical questions which arose during the writing of this manuscript. 

\bibliographystyle{plain}
{\footnotesize
\bibliography{mybib}
}

\end{document}